\documentclass[final,hidelinks,onefignum,onetabnum]{siamart190516}

\usepackage{mmap}

\usepackage{graphicx,amsmath,amsfonts,amssymb}
\usepackage{url,hyperref,algorithm,subcaption}
\usepackage{algpseudocode}
\usepackage{xcolor}
\graphicspath{{figures/}}

\usepackage[left=7.5em, right=7.5em, top=9.3em, bottom=9.3em]{geometry}

\newcommand{\BEAS}{\begin{eqnarray*}}
\newcommand{\EEAS}{\end{eqnarray*}}
\newcommand{\BEQ}{\begin{equation}}
\newcommand{\EEQ}{\end{equation}}
\newcommand{\BIT}{\begin{itemize}}
\newcommand{\EIT}{\end{itemize}}

\newcommand{\eg}{{\rm e.g.}}
\newcommand{\ie}{{\rm i.e.}}
\newcommand{\cf}{{\rm cf. }}
\newcommand{\ones}{\mathbf 1}

\newcommand{\reals}{{\mbox{\bf R}}}
\newcommand{\symm}{{\mbox{\bf S}}}

\newcommand{\dist}{\mathop{\bf dist{}}}
\newcommand{\argmin}{\mathop{\rm argmin}}

\newcommand{\range}{\mathop{\bf range}}

\newcommand{\relint}{\mathop{\bf relint}}
\newcommand{\dom}{\mathop{\bf dom}}

\newcommand{\prox}{{\mathop{\textbf{prox}}}}

{\begin{quote}}{\end{quote}}

\makeatletter
\long\def\@makecaption#1#2{
   \vskip 9pt
   \begin{small}
   \setbox\@tempboxa\hbox{{\bf #1:} #2}
   \ifdim \wd\@tempboxa > 5.5in
        \begin{center}
        \begin{minipage}[t]{5.5in}
        \addtolength{\baselineskip}{-0.95pt}
        {\bf #1:} #2 \par
        \addtolength{\baselineskip}{0.95pt}
        \end{minipage}
        \end{center}
   \else
        \hbox to\hsize{\hfil\box\@tempboxa\hfil}
   \fi
   \end{small}\par
}
\makeatother

\newcounter{oursection}

\newcounter{lecture}

\ifpdf
\hypersetup{
  pdftitle={A2DR},
  pdfauthor={A. Fu, J. Zhang, S. Boyd}
}
\fi

\renewcommand{\thepage}{\arabic{page}}

\newsiamremark{remark}{Remark}

\makeatletter
\long\def\@makecaption#1#2{%
    \footnotesize
    \setlength{\parindent}{1.5pc}
  \ifx\@captype\@figtxt
    \vskip\abovecaptionskip
    \setbox\@tempboxa\hbox{{\normalfont\scshape #1}. {\normalfont\itshape #2}}%
    \ifdim \wd\@tempboxa >\hsize
      {\normalfont\scshape #1}. {\normalfont\itshape #2}\par
    \else
      \global\@minipagefalse
      \hb@xt@\hsize{\hfil\box\@tempboxa\hfil}%
    \fi
  \else
    \hbox to\hsize{\hfil{\normalfont\scshape #1}\hfil}%
    \setbox\@tempboxa\hbox{{\normalfont\itshape #2}}%
    \ifdim \wd\@tempboxa >\hsize
      {\normalfont\itshape #2}\par
    \else
     \global\@minipagefalse
      \hb@xt@\hsize{\hfil\box\@tempboxa\hfil}%
    \fi
    \vskip\belowcaptionskip
  \fi}
\makeatother  

\title{Anderson Accelerated Douglas--Rachford Splitting\thanks{
Anqi Fu and 
Junzi Zhang contributed equally to this work.\URL{10.1137/19M1290097}}\funding{The work of the first and second authors was each supported by a Stanford Graduate Fellowship.}}

\author{Anqi Fu\thanks{Department of Electrical Engineering, Stanford University, Stanford, CA 94305 USA
(\href{mailto:anqif@stanford.edu}{anqif@} \href{mailto:anqif@stanford.edu}{stanford.edu}, 
\email{boyd@stanford.edu}).}\orcid{https://orcid.org/0000-0002-2876-2942}
\and Junzi Zhang\thanks{ICME, Stanford University, Stanford, CA 94305 USA (\email{junziz@stanford.edu}).}\orcid{https://orcid.org/0000-0002-5086-0063}
\and Stephen Boyd\footnotemark[2]}

\headers{ANDERSON ACCELERATED DOUGLAS--RACHFORD SPLITTING}{Anqi Fu, Junzi Zhang, and Stephen P. Boyd}

\begin{document}

\maketitle

\begin{abstract}
We consider the problem of nonsmooth convex optimization with linear equality constraints, 
where the objective function is only accessible through its proximal operator. This 
problem arises in many different fields such as statistical learning, computational 
imaging, telecommunications, and optimal control. To solve it, we propose an Anderson 
accelerated Douglas--Rachford splitting (A2DR) algorithm, which we show either globally 
converges or provides a certificate of infeasibility/unboundedness under very mild 
conditions. Applied to a block separable objective, A2DR partially decouples so that 
its steps may be carried out in parallel, yielding an algorithm that is fast and scalable 
to multiple processors.  We describe an open-source implementation and demonstrate its 
performance on a wide range of examples.
\end{abstract}



\begin{DOI}
10.1137/19M1290097
\end{DOI}

\section{Introduction}\label{intro}
\subsection{Problem setting}
Consider the convex optimization problem
\begin{equation}\label{general_n}
\begin{array}{ll}
\text{minimize} & f(x)\\
\text{subject to} & Ax=b
\end{array}
\end{equation}
with variable $x \in \reals^n$, where $f:\reals^n\rightarrow\reals
\cup\{+\infty\}$ is convex, closed, and proper (CCP), and $A \in 
\reals^{m\times n}$ and $b\in\reals^m$ are given. We assume that 
the linear constraint $Ax=b$ is feasible.

\paragraph{Block form} In this paper, we work with block separable 
$f$, i.e., $f(x)=\sum_{i=1}^Nf_i(x_i)$ for individually CCP $f_i:
\reals^{n_i}\rightarrow\reals\cup\{+\infty\}$, $i=1,\dots,N$. We 
partition $x=(x_1,\dots,x_N)$ so that $n=\sum_{i=1}^Nn_i$ and let 
$A=[A_1~A_2~\cdots~A_N]$ with $A_i\in\reals^{m\times n_i}$, $i=1,
\dots,N$. Problem (\ref{general_n}) can be written in terms of the
block variables as
\begin{equation}\label{general}
\begin{array}{ll}
\text{minimize} & \sum_{i=1}^Nf_i(x_i)\\
\text{subject to} & \sum_{i=1}^NA_ix_i=b.
\end{array}
\end{equation}
Many interesting problems have the form (\ref{general}), 
such as consensus optimization \cite{ADMM} and cone programming 
\cite{SCS}. In fact, by transforming nonlinear convex constraints 
(e.g., cone constraints) into set indicator functions and adding 
them to the objective function, any convex optimization problem 
can be written in the above form.

\paragraph{Optimality conditions}  
The point $x \in \reals^n$ is a solution to (\ref{general}) if there 
exist $g\in\reals^n$ and  $\lambda\in\reals^m$ such that 
\begin{equation}\label{primal}
Ax=b,
\end{equation}
\begin{equation}\label{dual}
0= g+A^T\lambda,\quad g\in\partial f(x),
\end{equation}
where $\partial f(x)$ is the subdifferential of $f$ at $x$. With block 
separability, (\ref{dual}) can be written as 
\[
0=g_i+A_i^T\lambda, \quad 
g_i\in\partial f_i(x_i), \quad i=1,\dots,N.
\]
We refer to (\ref{primal}) and (\ref{dual}) as the primal feasibility and 
dual feasibility conditions, and $x$ and $\lambda$ as the primal 
variable and dual variable, respectively. Together, these conditions 
are sufficient for optimality; they become necessary 
as well when Slater's constraint qualification is satisfied, i.e., 
$\relint\dom f\cap\{x\,:\,Ax=b\}\neq \emptyset$.

\paragraph{Proximal oracle} 
Methods for solving (\ref{general}) vary depending on what oracle
is available for $f_i$.
If $f_i$ and its subgradient can be queried directly, 
a variety of iterative algorithms may be used 
\cite{BoydVandenberghe:2004, NocedalandWright:2006, Nesterov:2013}. 
However, in our setting, we assume that each $f_i$ can only be accessed 
through its proximal operator $\prox_{tf_i}:\reals^{n_i}
\rightarrow\reals^{n_i}$, defined as
\[
\prox_{tf_i}(v_i)=\argmin\nolimits_{x_i}~ \left(f_i(x_i)+
\tfrac{1}{2t}\|x_i-v_i\|_2^2\right), 
\]
where $t > 0$ is a parameter. In particular, we assume neither 
direct access to the function $f_i$ nor its subdifferential 
$\partial f_i$. The separability of $f$ implies that \cite{Proximal}
\[
\prox_{tf}(v)=\left(\prox_{tf_1}(v_1),\dots,
\prox_{tf_N}(v_N)\right)
\]
for any $v=(v_1,\dots, v_N) \in \reals^n$. 

While we cannot evaluate $\partial f_i$ at a general point,
we can find an element of $\partial f_i$ at the proximal operator's image point:
\[
x_i=\prox_{tf_i}(v_i) \Longleftrightarrow 0\in\partial 
f_i(x_i)+\tfrac{1}{t}(x_i-v_i)\Longleftrightarrow \tfrac{1}{t}(v_i-x_i)\in\partial f_i(x_i).
\] 
Thus, by querying the proximal oracle of $f_i$ at $v_i$, we obtain 
an element in the subgradient of $f_i$ at $x_i = \prox_{tf_i}(v_i)$.

The optimality conditions can be expressed using the proximal operator as 
well. The point $x \in \reals^n$ is a solution to (\ref{general}) if there exist
$v\in\reals^n$ and $\lambda\in\reals^m$ such that
\begin{equation}\label{primal_prox}
Ax=b,
\end{equation}
\begin{equation}\label{dual_prox}
0=\tfrac{1}{t}(v-x)+A^T\lambda,\quad x_i=\prox_{tf_i}(v_i),\quad i=1,\dots,N.
\end{equation}

\paragraph{Residuals} 
From conditions (\ref{primal_prox}) 
and (\ref{dual_prox}), we define
the primal and dual residuals at $x, \lambda$ as
\begin{equation}\label{primal_res}
r_{\rm{prim}}=Ax-b,
\end{equation}
\begin{equation}\label{dual_res}
r_{\rm{dual}}=\tfrac{1}{t}(v-x)+A^T\lambda,
\end{equation}
and we define the overall residual as $r=(r_{\rm prim},r_{\rm dual}) \in \reals^{n+m}$.

\paragraph{Stopping criterion} 
If problem (\ref{general}) is feasible 
and bounded, a reasonable stopping criterion is that the residual 
norm lies below some threshold, \ie, $\|r\|_2\leq \epsilon_{\rm tol}$,
where $\epsilon_{\rm tol} > 0$ is a user-specified tolerance.
We refer to the associated $x$ as an approximate solution to (\ref{general}). 
We defer discussion of the criteria for 
pathological (infeasible/unbounded) cases to section \ref{global_theory}.

Notice that given a candidate $v \in \reals^n$, we can readily choose
the primal point $x=\prox_{tf}(v)$ and dual point
\begin{equation}
\lambda=\tfrac{1}{t}(A^\dagger)^T(x-v)\in\argmin\nolimits_{\hat\lambda}~ \|A^T\hat\lambda
-\tfrac{1}{t}(x-v)\|_2,
\label{ls_lam}
\end{equation}
a minimizer of the dual residual norm, where $A^{\dagger}$ denotes the
pseudoinverse of $A$. Thus, any algorithm for solving (\ref{general}) 
via the proximal oracle need only determine a $v$ that produces a small 
residual norm.

\subsection{Related work}\label{related_work} 
When functional access is restricted to a proximal oracle, the most common approaches 
for solving (\ref{general}) are the alternating direction method of multipliers (ADMM) \cite{Distr-ADMM, BlockSplit, POGS, DarrenDistr}, 
Douglas--Rachford splitting (DRS) \cite{DRS_multiblock}, and the augmented Lagrangian method 
\cite{cons_ALM} with appropriate problem reformulations (\eg, consensus). 
These algorithms take advantage of the separability of the objective function, making them 
well-suited for the nonsmooth convex optimization problem considered in this paper. 
Yet despite their robustness and scalability, they typically suffer 
from slow convergence. Researchers have proposed several acceleration techniques, 
including adaptive penalty parameters \cite{RB2000, Xu2017b}, adaptive 
synchronization \cite{Ada_Sync}, and momentum methods \cite{HB_ODE_Distr}. 
In practice, improvement from these techniques is usually limited due to the 
first-order nature of the accelerated algorithms. Special cases of (\ref{general}) can 
sometimes yield exploitable problem forms, such as the Laplacian regularized stratified 
model in \cite{cvx_strat}. There the authors use the structure of the Laplacian matrix to 
efficiently parallelize ADMM. However, for the general problem, further acceleration requires a
quasi-Newton method with line search \cite{SuperMann} or semismooth Newton method 
with access to the Clarke's generalized Jacobian of the objective's proximal operator 
\cite{ZicoSemNewton,ZaiwenSemNewton,milzarek2019stochastic}, both of which typically 
impose high per-iteration costs and memory requirements.

The acceleration technique adopted in this paper, type-II Anderson acceleration (AA), 
dates back to the 1960s \cite{AA}. It belongs to the family of sequence acceleration methods, 
which achieve faster convergence through certain sequence transformations. The origin of 
these methods can be traced to Euler's transformation of series \cite{abramowitz1948handbook}
from the 18th century. Several faster sequence acceleration techniques were proposed in the 20th 
century, including Aitken's $\Delta^2$-process in 1926 \cite{aitken1927xxv} along with its 
higher-order \cite{shanks1955non,wynn1956device} and vector \cite{mevsina1977convergence,macleod1986acceleration,smith1987extrapolation} extensions, 
of which AA is a member. We refer readers to \cite{brezinski2018shanks,bachacceleration} for 
a thorough history. AA can be viewed as either an extrapolation method or a generalized 
quasi-Newton method \cite{FangSecant}. However, unlike classical quasi-Newton methods, it is 
effective without a line search and requires less computation and memory per iteration so long 
as certain stabilization measures are adopted. 

Type-II AA was initially proposed to accelerate solvers for nonlinear integral equations in 
computational chemistry and materials science; later, it was applied to general fixed-point 
problems \cite{WalkerNi}. It operates by using an affine combination of previous iterates to 
determine the next iterate of an algorithm, where the combination's coefficients are obtained 
by solving an unconstrained least squares problem. In this sense, it is a generalization of the 
averaged iteration algorithm and Nesterov's accelerated gradient method. Its local convergence 
properties have been analyzed in a range of settings, both deterministic 
\cite{ConvDIIS, ConvAA, RNA, sara2018, aa_cheby, aa_constr, sara2019} and stochastic 
\cite{SRNA, ConvAA_rand}, but its global convergence properties remain largely unknown 
except for a variant called EDIIS \cite{KudinScuseria:2002}. EDIIS has been shown to converge 
globally assuming that the fixed-point mapping is contractive \cite{ChenKelley:2019}. However, 
it adds nonnegativity constraints to the coefficients of AA, meaning each iteration must solve 
a nonnegative least squares problem, a more complex task than solving the unconstrained problem, 
which admits a closed-form solution. The technique proposed in this paper, by contrast, only 
requires nonexpansiveness for global convergence. Each of its iterations merely solves an 
unconstrained least squares problem, similar to the original type-II AA.
Recently, \cite{FangSecant} proposed another AA variant called type-I AA. While less stable than 
its type-II counterpart, this variant performs more favorably with appropriate stabilization and 
globalization \cite{SCS-2-1-2,AA1}.

AA has been applied in the literature to several problems 
related to (\ref{general}). The authors of \cite{AA-GP} use AA to speed up 
a parallelized local-global solver for geometry optimization and physics 
simulation problems, which may be viewed as a special case of our problem 
where $f_i$ are projection operators. In a separate setting, \cite{Parallel2}
employs AA to solve large-scale fixed-point problems arising from partial 
differential equations, demonstrating performance improvements on a 
distributed memory platform. 
More generally, \cite{AA1} uses type-I AA in conjunction with 
DRS and a splitting conic solver (SCS) \cite{SCS} to solve problems in 
consensus and conic optimization. These results are extended by \cite{SuperSCS}, 
which combines type-II AA with an SCS variant to produce SuperSCS, an 
efficient solver for large cone programs. 
AA has also seen success in nonconvex settings. Notably, \cite{aa_admm} applies AA 
to ADMM and studies its empirical performance on nonconvex optimization 
problems arising in computer graphics.

\subsection{Contribution}\label{contrib}
In this paper, we consider the DRS algorithm for solving (\ref{general}), 
which satisfies the proximal oracle assumption and admits a simple fixed-point 
(FP) formulation \cite{MonoPrimer}. This FP format allows us to improve the convergence of 
DRS with AA, a memory efficient, line search free acceleration method that works on generic 
nonsmooth, nonexpansive FP mappings with almost no extra cost per iteration \cite{AA1}. 
Motivated by the need for solver stability, we choose type-II AA in our current work and 
propose a robust stabilization scheme that maintains its speed and efficiency. 
We then apply it to DRS and show that the resulting Anderson accelerated Douglas--Rachford 
splitting (A2DR) algorithm always either converges or provides a certificate of 
infeasibility/unboundedness under very relaxed conditions. As a consequence, we obtain the 
first globally convergent type-II AA variant in nonsmooth, potentially pathological settings. 
Our convergence analysis only requires nonexpansiveness of the FP mapping,  
gracefully handling cases when a fixed-point does not exist. Finally, we release an 
open-source Python solver based on A2DR at
\smallskip
\begin{center}
\url{https://github.com/cvxgrp/a2dr}.
\end{center}
\smallskip

\paragraph{Outline} We begin in section \ref{DRS} by introducing the basics
of DRS. We then describe AA and propose A2DR in section \ref{A2DR}. The 
global convergence properties of A2DR are established in section \ref{global_theory},
along with an analysis of the infeasible and unbounded cases. We discuss the 
presolve, equilibration, and hyperparameter choices in section \ref{precond}, 
followed by the implementation details in section \ref{solver}. In 
section \ref{experiments}, we demonstrate the performance of A2DR on several examples. 
We conclude in section \ref{conclusions}.

\section{Douglas--Rachford splitting}\label{DRS}
Douglas--Rachford splitting (DRS) is an algorithm for solving 
problems of the form
\[
\mbox{minimize} \quad g(x) + h(x)
\]
with variable $x$, where $g$ and $h$ are CCP \cite{MonoPrimer}.
We can write problem (\ref{general}) in this form
by taking $g=f$ and $h=\mathcal{I}_{\{x\,:\,Ax=b\}}$, the indicator function
of the linear equality constraint.
Notice that $\prox_{t h}$ is the projection onto the associated subspace, defined as
\[
\Pi(v^{k+1/2})=v^{k+1/2}-A^{\dagger}(Av^{k+1/2}-b)=v^{k+1/2}-
A^T(AA^T)^{\dagger}(Av^{k+1/2}-b).
\]
The DRS algorithm proceeds as follows.

{\linespread{1.29}
\begin{algorithm}[h]
\caption{\rm{Douglas--Rachford Splitting (DRS)}}
\label{DRS_alg}
\begin{algorithmic}[1]
\State {\bfseries Input:} initial point $v^0$, penalty coefficient 
$t>0$. 
\For{$k=1, 2, \dots$}
\State $x^{k+1/2}=\prox_{tf}(v^k)$ 
\State $v^{k+1/2}=2x^{k+1/2}-v^k$
\State $x^{k+1}=\Pi(v^{k+1/2})$
\State $v^{k+1}=v^k+x^{k+1}-x^{k+1/2}$
\EndFor
\end{algorithmic}
\end{algorithm}
}

Each iteration $k$ requires the evaluation of the proximal operator of $f$
and the projection onto a linear subspace.

\paragraph{Dual variable and residuals}
We regard $x^{k+1/2}$, the proximal operator's image point, 
as our approximate primal optimal variable in iteration $k$.
There are two ways to produce an approximate dual variable $\lambda^k$.
The first way sets 
\[
\lambda^k = \tfrac{1}{t}(AA^T)^{\dagger}(Av^{k+1/2} - b), 
\]
an intermediate value from the projection step. (See Remark \ref{rmk_app} in 
the supplementary materials for the reasoning behind this choice.)
The second way computes $\lambda^k$ as the minimizer of $\|r_{\rm dual}^k\|_2$,
which necessitates solving the least squares problem (\ref{ls_lam}) at each
iteration.  
Our implementation uses the second
method because the additional computational cost is minimal,
and this choice of a dual optimal variable results in earlier stopping.

The primal and dual residuals can be calculated by plugging our DRS iterates
into (\ref{primal_res}) and (\ref{dual_res}):
\begin{equation}\label{primal_k}
r_{\rm prim}^k=Ax^{k+1/2}-b,
\end{equation}
\begin{equation}\label{dual_k_ls}
r_{\rm dual}^k=\tfrac{1}{t}(v^k-x^{k+1/2})+A^T\lambda^k.
\end{equation}

\paragraph{Convergence}
Define the fixed-point mapping $F_{\text{DRS}}:\reals^n\rightarrow\reals^n$ as 
\[
F_{\text{DRS}}(v)=v+\Pi\left(2\textbf{prox}_{tf}(v)-v\right)-\textbf{prox}_{tf}(v),
\]
so that $v^{k+1}=F_{\text{DRS}}(v^k)$.
It can be shown that $F_{\rm DRS}$ is $1/2$-averaged (\ie, 
$F_{\rm DRS}=\frac{1}{2}H+\frac{1}{2}I$, where $H$ is nonexpansive and $I$ is the 
identity mapping), and hence, $v^k$ converges globally and sublinearly to a 
fixed-point of $F_{\rm DRS}$ whenever such a point exists. In this case, 
$x^{k+1/2}$ and $x^{k+1}$ both converge to a solution of (\ref{general}), implying that 
$\lim_{k\rightarrow\infty} \|r_{\rm prim}^k\|_2 = \lim_{k\rightarrow\infty} 
\|r_{\rm dual}^k\|_2 = 0$ \cite{MonoPrimer}.

\section{Anderson accelerated DRS}\label{A2DR}
In this section, we give a brief overview of AA and propose a modification that improves its stability. 
We then combine stabilized AA with DRS to construct our main algorithm,
Anderson accelerated DRS. A2DR always produces 
an approximate solution to \eqref{general} when the problem is feasible 
and bounded. We treat the infeasible/unbounded cases in section \ref{global_theory}.

\subsection{Anderson acceleration} 
Consider a $1/2$-averaged mapping $F:\reals^n\rightarrow\reals^n$. 
To solve the associated fixed-point problem $F(v)=v$, we can repeatedly 
apply the fixed-point iteration (FPI) $v^{k+1}=F(v^k)$, which is exactly  
DRS when $F=F_{\rm DRS}$. However, convergence of FPI algorithms is 
usually slow in practice. Acceleration schemes are one way of 
addressing this flaw. 
AA is a special form of the generalized limited-memory quasi-Newton (LM-QN)
method. It is one of the most successful acceleration schemes for 
general nonsmooth FPIs, exhibiting greater memory efficiency than 
classical LM-QN algorithms like the restarted Broyden's method \cite{SuperSCS}.

We focus here on the original type-II AA \cite{AA}. Let $G(v)=v-F(v)$ 
be the residual function and $M^k \in \mathbf{Z}_+$ a nonnegative integer 
denoting the memory size. Typically, $M^k=\min(M_{\max},k)$ for some 
maximum memory $M_{\max} \geq 1$ \cite{WalkerNi}. 
At iteration $k$, type-II AA stores in memory the most recent $M^k+1$ 
iterates $(v^k,\dots,v^{k-M^k})$ and replaces $v^{k+1}=F(v^k)$ with 
$v^{k+1}=\sum_{j=0}^{M^k}\alpha_j^kF(v^{k-M^k+j})$, where $\alpha^k
=(\alpha_0^k,\dots,\alpha_{M^k}^k)$ is the solution to 
\begin{equation}\label{aa2sub}
\begin{array}{ll}
\mbox{minimize} & \|\sum_{j=0}^{M^k}\alpha_j^k G(v^{k-M^k+j})\|_2^2\\
\mbox{subject to} & \sum_{j=0}^{M^k}\alpha_j^k=1.
\end{array}
\end{equation}
AA then updates its memory to $(v^{k+1},\dots,v^{k+1-M^{k+1}})$ before
repeating the process.

The accelerated $v^{k+1}$ can be seen as an extrapolation from the original
$v^{k+1}$ and the fixed-point 
mappings of a few earlier iterates. It has 
the potential to reduce the residual by a significant amount. In particular, 
when $F$ is affine, (\ref{aa2sub}) seeks an 
affine combination $\tilde{v}^{k+1}$ of the last $M^k+1$ 
iterates that minimizes the residual norm $\|G(\tilde{v}^{k+1})\|_2$, 
then computes $v^{k+1}=F(\tilde{v}^{k+1})$ by performing 
an additional FPI. 

\subsection{Main algorithm}
Despite the popularity of type-II AA, it suffers from instability in its 
original form \cite{RNA}. We propose a stabilized variant using adaptive 
regularization and a simple safeguarding globalization trick. 

\paragraph{Adaptive regularization}
Define $g^k=G(v^k)$, $y^k=g^{k+1}-g^k$, $s^k=v^{k+1}-v^k$, 
$Y_k=[y^{k-M^k}~\cdots~y^{k-1}]$, and $S_k=[s^{k-M^k}~\cdots~s^{k-1}]$. 
With a change of variables, (\ref{aa2sub}) can be rewritten
as \cite{WalkerNi}
\begin{equation}\label{aa2sub-simp}
\mbox{minimize} \quad \|g^k-Y_k\gamma^k\|_2
\end{equation}
with respect to $\gamma^k=(\gamma_0^k,\dots,\gamma_{M^k-1}^k)$, where 
\begin{equation}\label{alpha-gamma}
\alpha_0^k=\gamma_0^k, \quad \alpha_i^k=
\gamma_i^k-\gamma_{i-1}^k, \quad i = 1,\ldots,M^k-1, \quad \alpha_{M^k}^k=
1-\gamma_{M^k-1}^k.
\end{equation}

To improve stability, we add an $\ell_2$-regularization term to 
(\ref{aa2sub-simp}), scaled by the Frobenius norms of $S_k$ and $Y_k$, 
which yields the problem
\begin{equation}\label{unc_aa2sub}
\begin{array}{ll}
\text{minimize} & \|g^k-Y_k\gamma^k\|_2^2+\eta\left(\|S_k\|_F^2+
\|Y_k\|_F^2\right)\|\gamma^k\|_2^2,
\end{array}
\end{equation}
where $\eta > 0$ is a parameter. The regularization adopted in (\ref{unc_aa2sub}) 
differs from the one introduced in \cite{RNA} that directly regularizes $\alpha^k$.  
We argue that with the affine constraint on $\alpha^k$, it is more natural
to regularize the unconstrained variables $\gamma^k$. This approach also 
allows us to establish global convergence in section \ref{global_theory}. 
Intuitively, if the algorithm is converging, 
$\lim_{k\rightarrow\infty} \|S_k\|_F = \lim_{k\rightarrow\infty} \|Y_k\|_F = 0$, 
so the coefficient on the regularization term vanishes just like in the single iteration 
local analysis by \cite{RNA}.

\paragraph{A simple and relaxed safeguard} To achieve global convergence, 
we also need a safeguarding step. This step checks whether the current residual 
norm is sufficiently small. If true, the algorithm takes the AA update and 
skips the safeguarding check for the next $R-1$ iterations. 
Otherwise, the algorithm replaces the AA update with the vanilla FPI update. 
Here $R \in \mathbf{Z}_{++}$ is a positive integer that determines the degree of 
safeguarding; smaller values are more conservative, since the safeguarding step is 
performed more often.

\paragraph{A$2$DR} We are finally ready to present A2DR (Algorithm \ref{A2DR_alg}).
A2DR applies type-II AA with adaptive regularization (lines 10--11) and safeguarding 
(lines 13--17) to the DRS fixed-point mapping $F_{\rm DRS}$. 
In our description, $G_{\rm DRS} = I - F_{\rm DRS}$ is the residual mapping, 
$D>0,~\epsilon>0$ are constants that characterize the degree of safeguarding, 
and $n_{\rm AA}^k$ is the number of times the $AA$ candidate has passed the 
safeguarding check up to iteration $k$. 

\paragraph{Stopping criterion}
As explained in section \ref{intro}, to check optimality, we evaluate the 
primal and dual residuals $r_{\rm prim}^k$ and $r_{\rm dual}^k$. 
We terminate the algorithm and output $x^{k+1/2}$ as the approximate solution if 
\begin{equation}\label{opt_detect}
\|r^k\|_2\leq\epsilon_{\rm tol}=\epsilon_{\rm abs}+\epsilon_{\rm rel}\|r^0\|_2,
\end{equation}
where $r^k=(r_{\rm prim}^k, r_{\rm dual}^k)$ and $\epsilon_{\rm abs} > 0, 
\epsilon_{\rm rel}>0$ are user-specified absolute and relative tolerances, respectively. 

{\linespread{1.29}
\begin{algorithm}[h]
\caption{\rm{Anderson Accelerated Douglas--Rachford Splitting (A2DR)}}
\label{A2DR_alg}
\begin{algorithmic}[1]
\State {\bfseries Input:} initial point $v^0$, penalty coefficient $t>0$, 
regularization coefficient $\eta>0$,  safeguarding constants 
$D>0,~\epsilon>0, ~R \in \mathbf{Z}_{++}$, max-memory $M_{\max} \in \mathbf{Z}_+$. 
\State Initialize $n_{\rm AA}=0,~R_{\rm AA}=0,~I_{\text{safeguard}}=\texttt{True}$. 
\State Compute $v^1=F_{\rm DRS}(v^0)$, $g^0=v^0-v^1$.
\For{$k=1, 2, \dots$}
\State \texttt{\# Memory update}
\State Choose memory $M^k=\min(M_{\max}, k)$. 
\State Compute the DRS candidate: 
$v_{\rm DRS}^{k+1}=F_{\rm DRS}(v^k)$, $g^k=v^k-v^{k+1}_{\rm DRS}$.
\State Update $Y_k$ and $S_k$ with $y^{k-1}=g^k-g^{k-1}$ and $s^{k-1}=v^k-v^{k-1}$.
\State \texttt{\# Adaptive regularization}
\State Solve for $\gamma^k$ in regularized least squares (\ref{unc_aa2sub}) 
and compute weights $\alpha^k$ from (\ref{alpha-gamma}).
\State Compute the AA candidate: 
$v_{\rm AA}^{k+1}=\sum_{j=0}^{M^k}\alpha_j^kv_{\rm DRS}^{k-M^k+j+1}$.
\State \texttt{\# Safeguard}
\State {\bf If} $I_{\text{safeguard}}$ is \texttt{True} or $R_{\rm AA}\geq R$:
\State \hspace{0.5cm} \textbf{If} 
$\|g^k\|_2=\|G_{\rm DRS}(v^k)\|_2\leq D\|g^0\|_2(n_{\rm AA}/R+1)^{-(1+\epsilon)}$:
\State \hspace{1.2cm} $v^{k+1}=v_{\rm AA}^{k+1}$, $n_{\rm AA}=n_{\rm AA}+1$, 
$I_{\text{safeguard}}=\texttt{False}$, $R_{\rm AA}=1$.
\State \hspace{0.5cm} \textbf{else} $v^{k+1}=v_{\rm DRS}^{k+1}$, $R_{\rm AA}=0$. 
\State \textbf{else} $v^{k+1}=v_{\rm AA}^{k+1}$, $n_{\rm AA}=n_{\rm AA}+1$, 
$R_{\rm AA}=R_{\rm AA}+1$.
\State Terminate and output $x^{k+1/2}$ (\cf Algorithm \ref{DRS_alg}) 
if stopping criterion (\ref{opt_detect}) is satisfied. 
\EndFor
\end{algorithmic}
\end{algorithm}
}

\section{Global convergence}\label{global_theory}
We now establish the global convergence properties of A2DR. In particular, 
we show that under the general assumptions in section \ref{intro}, 
A2DR either converges globally from any initial point or provides a certificate 
of infeasibility/unboundedness.

\subsection{Infeasibility and unboundedness}\label{global_theory:pathology}
When the optimality conditions do not hold even in the asymptotic sense, \ie, 
if the infimum of the primal or dual residual over all possible $x$ and $v$ is 
nonzero, problem (\ref{general}) is either infeasible or unbounded. We 
say that (\ref{general}) is \textit{infeasible} if 
$\dom f\cap \{x\,:\,Ax=b\}=\emptyset$, and we say that it is \textit{unbounded}
if (\ref{general}) is feasible, but $\inf_{Ax=b} f(x)=-\infty$.  
The following proposition characterizes sufficient certificates of 
infeasibility and unboundedness.

\begin{proposition}[certificates of infeasibility and unboundedness]
Let $f^*:\reals^n\rightarrow\reals\cup\{+\infty\}$ denote the conjugate 
function of $f$, defined as $f^*(y)=\sup\nolimits_{x\in\dom f}\big(y^Tx-  f(x)\big)$.
\begin{enumerate}
\item[{\rm (i)}] If $\dist(\dom f, \,\{x\,:\,Ax=b\})>0$, then problem \eqref{general} 
is infeasible.
\item[{\rm (ii)}] If $\dist(\dom f^*,\,\range(A^T))>0$, then problem \eqref{general} 
is unbounded.
\end{enumerate}
\end{proposition}

When (i) holds, (\ref{general}) is also called \textit{(primal) 
strongly infeasible}, and when (ii) holds, (\ref{general}) is called 
\textit{dual strongly infeasible} \cite{DRS_newuse}. We say that 
(\ref{general}) is \textit{pathological} if it is either primal or 
dual strongly infeasible, and \textit{solvable} otherwise. Notice that 
when the problem is pathological, it is either infeasible or unbounded, 
but not both.

\begin{proof}
Claim (i) is true by definition. To prove claim (ii), observe that 
the dual problem of (\ref{general}) is 
$\text{minimize}_{\nu}~f^*(\nu)+g^*(-\nu)$, where $g^*(\nu)=b^T\lambda$ 
when $\nu=A^T\lambda$, and $g^*(\nu)=+\infty$ otherwise. By Lemma 1 
in \cite{DRS_pathology}, if $\dist(\dom f^*,\,\range(A^T))>0$, 
then the dual problem is strongly infeasible, and hence the primal problem 
(\ref{general}) is unbounded.
\end{proof}

If (\ref{general}) is pathological, an algorithm should provide a 
certificate of either (i) or (ii). We will show that A2DR achieves 
this goal by returning the distances in (i) and (ii) as a by-product of 
its iterations.

\subsection{Convergence results}
We are now ready to present the convergence results for A2DR. We begin by 
highlighting the contribution of adaptive regularization to the stabilization 
of AA. Indeed, by setting the gradient of the objective function in 
(\ref{unc_aa2sub}) to zero, we find the solution is 
\[
\gamma^k=(Y_k^TY_k+\eta \left(\|S_k\|_F^2+\|Y_k\|_F^2\right)I)^{-1}Y_k^Tg^k.
\]
Using the relationship between $\alpha^k$ and $\gamma^k$, we then write
\begin{equation}\label{xkupdate}
\begin{split}
v^{k+1}&=v^k-(I+(S_k-Y_k)(Y_k^TY_k+\eta 
\left(\|S_k\|_F^2+\|Y_k\|_F^2\right)I)^{-1}Y_k^T)g^k=v^k-H_kg^k,
\end{split}
\end{equation}
where 
$H_k=I+(S_k-Y_k)(Y_k^TY_k+\eta \left(\|S_k\|_F^2+
\|Y_k\|_F^2\right)I)^{-1}Y_k^T$. 

\begin{lemma}\label{Hkbound}
The matrices $H_k$ ($k\geq 0$) satisfy $\|H_k\|_2
\leq 1+2/\eta$.
\end{lemma}
\begin{proof}
Since $\|A\|_2\leq \|A\|_F$ for any matrix $A$,
\begin{equation*}
\begin{split}
\|H_k\|_2&\leq 1+\dfrac{\|S_k-Y_k\|_2\|Y_k\|_2}{\eta 
\left(\|S_k\|_F^2+\|Y_k\|_F^2\right)}\leq 
1+\dfrac{\|S_k-Y_k\|_F\|Y_k\|_F}{\eta\left(\|S_k\|_F^2
+\|Y_k\|_F^2\right)}\\
&\leq 1+\dfrac{\|S_k\|_F\|Y_k\|_F+\|Y_k\|_F^2}{\eta
\left(\|S_k\|_F^2+\|Y_k\|_F^2\right)}
\leq 1+2/\eta.
\end{split}
\end{equation*}
This completes the proof.
\end{proof}

The above lemma characterizes the stability ensured by regularization 
in (\ref{unc_aa2sub}), providing a stepping stone to our global convergence 
theorems.

\subsubsection{Solvable case} 
\begin{theorem}\label{global_conv_well}
Suppose that problem \eqref{general} is solvable. Then for any 
initialization $v^0$ and any hyperparameters 
$\eta > 0,~D > 0,~\epsilon > 0,~R \in \mathbf{Z}_{++},~M_{\max} \in 
\mathbf{Z}_+$, we have 
\begin{equation}\label{liminf}
\liminf_{k\rightarrow\infty}\|r^k\|_2=0,
\end{equation}
and the AA candidates are adopted infinitely often.
Additionally, if $F_{\rm DRS}$ has a fixed-point, $v^k$ 
converges to a fixed-point of $F_{\rm DRS}$ and $x^{k+1/2}$ 
converges to a solution of \eqref{general} as $k \rightarrow \infty$.
\end{theorem}

The proof is left to the supplementary materials. A direct corollary of Theorem 
\ref{global_conv_well} is that the primal and dual residuals of 
$x^{k+1/2}$ converge to zero so long as (\ref{general}) 
is feasible and bounded. 
Even if (\ref{general}) does not have a solution, A2DR still 
produces a sequence of asymptotically optimal points provided that \eqref{general}
is not pathological. 
Thus, Algorithm \ref{A2DR_alg}
always terminates in a finite number of steps in these cases.

In practice, the proximal operators and projections 
are often evaluated with error, so lines 3 and 5 in Algorithm 
\ref{DRS_alg} become $\hat{x}^{k+1/2}=\prox_{tf}(v^k)+\zeta_1^k$ and 
$\hat{x}^{k+1}=\Pi(v^{k+1/2})+\zeta_2^k$, where 
$\zeta_1^k,~\zeta_2^k\in\reals^n$ represent numerical errors. 
We use $\hat{x}^{k+1/2}$, $\hat{v}^{k+1/2}$, $\hat{x}^{k+1}$ to 
denote the error-corrupted intermediate $F_{\rm DRS}$ iterates, 
and ${x}^{k+1/2}$, ${v}^{k+1/2}$, ${x}^{k+1}$ to denote the 
error-free intermediate $F_{\rm DRS}$ iterates.
However, we still use the old notation (\eg, $v^k$ and $g^k$) to 
denote the error-corrupted A2DR iterates in the body of Algorithm 
\ref{A2DR_alg}. For cases with such errors, we have the following 
convergence result. 

\begin{theorem}\label{global_conv_well_approx}
Suppose that problem \eqref{general} is solvable, but the $F_{\rm DRS}$ 
iterates are evaluated with errors $\zeta_1^k,~\zeta_2^k \in \reals^n$. 
Assume that $F_{\rm DRS}$ has a fixed-point and 
$\exists ~\epsilon'>0$ such that $\|\zeta_1^k\|_2\leq \epsilon'$ 
and $\|\zeta_2^k\|_2\leq \epsilon'$ for all $k\geq 0$. 
Then for any initialization $v^0$ and any hyperparameters 
$\eta > 0,~D > 0,~\epsilon > 0,~R \in 
\mathbf{Z}_{++},~M_{\max} \in \mathbf{Z}_+$, if all $v^k$ 
and some fixed-point $v^\star$ of $F_{\rm DRS}$ are uniformly 
bounded, \ie, $\|v^k\|_2\leq L$ and $\|v^\star\|_2\leq L$ for 
a constant $L>0$, we have 
\begin{equation}\label{liminf_err}
\liminf_{k\rightarrow\infty}\|r_{\rm prim}^k\|_2\leq 
\|A\|_2(4\epsilon'+4\sqrt{L\epsilon'}),\quad \liminf_{k\rightarrow\infty}
\|r_{\rm dual}^k\|_2\leq \tfrac{1}{t}(4\epsilon'+4\sqrt{L\epsilon'}).
\end{equation}
The residuals are computed by plugging $v^k$ (as output
by A\/$2$DR) and the error-free intermediate iterates $x^{k+1/2}=\prox_{tf}(v^k)$
into \eqref{primal_k} and \eqref{dual_k_ls}. 
\end{theorem}

\subsubsection{Pathological case} 
\hfill

\begin{theorem}\label{global_conv_path}
Suppose that problem \eqref{general} is pathological. Then for 
any initialization $v^0$ and any hyperparameters 
$\eta > 0,~D > 0,~\epsilon > 0,~R \in \mathbf{Z}_{++},~M_{\max} \in
\mathbf{Z}_+$, the difference $v^k-v^{k+1}$ converges to some
nonzero vector $\delta v\in\reals^n$. If, furthermore, 
$\lim_{k\rightarrow\infty}Ax^{k+1/2}=b$, then \eqref{general} 
is unbounded, in which case 
$\|\delta v\|_2=t\dist(\dom f^*, 
\range(A^T))$. Otherwise, $(\ref{general})$ is infeasible and 
$\|\delta v \|_2\geq\dist(\dom f,\break\{x\,:\,Ax=b\})$ 
with equality when the dual problem is feasible.
\end{theorem}

The proof is given in the supplementary materials. Theorem \ref{global_conv_path} 
states that in pathological cases, the successive differences 
$\delta v^k=v^k-v^{k+1}$ can be used as certificates of 
infeasibility and unboundedness. We leave 
the practical design and implementation of these certificates 
to a future version of A2DR.

The same global convergence results (Theorems 
\ref{global_conv_well}--\ref{global_conv_path}) can be shown 
for stabilized type-I AA \cite{AA1}, which sometimes exhibited 
better numerical performance in our early experiments.
However, type-I AA 
introduces additional hyperparameters, and to ensure our solver
is robust without the need for extra hyperparameter tuning, 
we restrict ourselves to type-II AA. We leave 
type-I Anderson accelerated DRS to a future paper.

\section{Presolve, equilibration, and parameter selection}
\label{precond}
In this section, we introduce a few tricks that make A2DR 
more efficient in practice. 

\paragraph{Infeasible linear constraints}  In section \ref{intro},
we assumed that the linear constraint $Ax=b$ is feasible. 
However, this assumption may be violated in practice. 
To address this issue, we first solve the least squares problem
associated with the linear system. If the resulting residual 
is sufficiently small, we proceed to solve (\ref{general}) 
using A2DR. Otherwise, we terminate and return a certificate of 
infeasibility.

\paragraph{Preconditioning} 
To precondition the problem, we scale the variables $x_i$ and the linear 
constraints (rows of $Ax=b$), solve the problem with the scaled variables 
and data, then unscale to recover the original variables. Scaling the 
variables and constraints does not change the theoretical convergence, 
but can improve the practical convergence if the scaling factors are 
chosen well. A popular heuristic for improving the practical convergence 
is to choose the scalings to minimize, or at least reduce, the condition 
number of the coefficient matrix. In turn, a heuristic for reducing the 
condition number of the coefficient matrix is to equilibrate it, \ie, 
choose the scalings so that all rows have approximately equal norm and 
all columns have approximately equal norm. The regularized Sinkhorn--Knopp 
method described below does this, where the regularization allows it to 
gracefully handle matrices that cannot be equilibrated or would require 
very extreme scaling to equilibrate.

The details are as follows.
First, we equilibrate $A$ by choosing diagonal matrices 
$D=\textbf{diag}(d_1,\dots,d_m)$ and $E=\textbf{diag}(e_1I_{n_1},
\dots,e_NI_{n_N})$, with $d_1>0,\dots,d_m>0$ and $e_1>0,\dots,e_N>0$,
and forming the scaled matrix $\hat{A}=DAE$.
The scaled problem is
\begin{equation}\label{general-rescaled}
\begin{array}{ll}
\text{minimize} & \sum_{i=1}^N\hat{f}_i(\hat{x}_i)\\
\text{subject to} & \sum_{i=1}^N\hat{A}_i\hat{x}_i
=\hat{b}, 
\end{array}
\end{equation}
where
\[
\hat{f}_i(\hat{x}_i)=f_i(e_i\hat{x}_i),\quad 
\hat{A}=D[A_1~A_2~\cdots~A_N]E,\quad \hat{b}=Db.
\]
We apply A2DR to (\ref{general-rescaled}) to obtain $\hat{x}^\star$ and recover the approximate 
solution to our original problem (\ref{general}) via $x^\star=E\hat{x}^\star$. 

To determine the scaling factors $d_i$ and $e_j$, we use the regularized
Sinkhorn--Knopp method \cite{POGS}. 
First, we perform a change of variables to $u_i = 2\log(d_i)$ and $v_j = 2\log(e_j)$.
Then we solve the optimization problem
\BEQ\label{equil_obj}
\begin{array}{ll}
\mbox{minimize} & \sum_{i=1}^m\sum_{j=1}^NB_{ij}e^{u_i+v_j}-
N{\bf 1}^Tu-m{\bf 1}^Tv+\gamma\left(N\sum_{i=1}^me^{u_i}
+m\sum_{j=1}^Ne^{v_j}\right)
\end{array}
\EEQ
for $u\in\reals^m$ and $v\in\reals^N$, where $B_{ij}=
\sum_{l=n_1+\cdots+n_{j-1}+1}^{n_1+\cdots+n_j}A_{il}^2$ and $\gamma > 0$ is a regularization parameter. 
This problem is strictly convex. At its solution, the arithmetic means of the recovered scaling 
factors are equal. In our implementation, we set 
\[
\gamma = \dfrac{m+N}{mN}\sqrt{\epsilon^{\textrm{mp}}},
\]
where $\epsilon^\textrm{mp}$ is the machine precision. 
Notice that when $\gamma=0$ and \eqref{equil_obj} has a solution, the resulting $\hat A$ is equilibrated exactly, 
\ie, the rows all have the same $\ell_2$ norm, and the columns all have the same $\ell_2$ norm in the blockwise sense 
(with block sizes $n_1,\dots,n_N$).
 
We use coordinate descent to solve \eqref{equil_obj}, which produces \cite[Algorithm 2]{POGS}. This algorithm typically 
returns a solution $\tilde{u}, \tilde{v}$ in only a handful of iterations. We then recover $\tilde{d}_i=e^{\tilde{u}_i/2}$ 
and $\tilde{e}_j=e^{\tilde{v}_j/2}$. Define  $\tilde{D}=\textbf{diag}(\tilde{d}_1,\dots,\tilde{d}_m)$ and 
$\tilde{E}=\textbf{diag}(\tilde{e}_1I_{n_1},\dots,\tilde{e}_NI_{n_N})$. 
Although the arithmetic means of $(\tilde{d}_1,\ldots,\tilde{d}_m)$ and $(\tilde{e}_1,\ldots,\break\tilde{e}_N)$ are already equal, 
we also wish to enforce equality of their geometric means, which corresponds to equality of the arithmetic means 
of the problem variables. This leads to better performance in practice.
Accordingly, we scale $\tilde{D}$ and $\tilde{E}$  to obtain $D$ and $E$ such that the geometric mean of $(d_1,\ldots,d_m)$ 
equals that of $(e_1,\ldots,e_N)$ and $\|DAE\|_F=\sqrt{\min(m,N)}$.

Since $E$ is constant within each variable block, the proximal 
operator of $\hat f_i$ can be evaluated using the proximal operator of $f_i$ via
\begin{equation}\label{prox_scaled}
\begin{split}
\hat{x}_i &= \prox_{t\hat{f}_i}(\hat{v}_i)
=\text{argmin}_{\hat{x}_i}~ \left(f_i(e_i\hat{x}_i)
+\tfrac{1}{2t}\|\hat{x}_i-\hat{v}_i\|_2^2\right)\\
&=\tfrac{1}{e_i}\text{argmin}_{x_i}~ \left(f_i(x_i)+\tfrac{1}{2t}
\|x_i/e_i-\hat{v}_i\|_2^2\right)\\
&=\tfrac{1}{e_i}\prox_{e_i^2 tf_i}(e_i\hat{v}_i).
\end{split}
\end{equation}
All other steps of A2DR (including the projection step in Algorithm
\ref{DRS_alg}, line 5) remain the same, except with $A$ and $b$ 
replaced by $\hat{A}$ and $\hat{b}$. 
We check the stopping criterion directly on \eqref{general-rescaled}, 
trusting that our equilibration scheme 
provides an appropriate scaling of the original problem.
An alternative is to check the stopping criterion on \eqref{general} using 
the unscaled variables.

\paragraph{Choice of $t$} With equilibration, the choice of parameter
\[
t=\frac{1}{10}\left(\prod_{j=1}^Ne_j\right)^{-2/N}
\]
works well across a wide variety of problems.
(Recall that convergence is guaranteed in theory for any $t>0$.)
Our implementation uses this choice of $t$.

The intuition behind our choice is as follows. 
Consider the case of $f_i(x_i)=x_i^TQ_ix_i$ 
with $Q_i\in\symm_+^{n_i}$, the set of symmetric positive semidefinite matrices. 
The associated $\prox_{tf_i}(v_i)=(2tQ_i+I)^{-1}v_i$ is linear, and by \eqref{prox_scaled},
\[
\prox_{t\hat{f}_i}(\hat{v}_i)=\prox_{e_i^2t f_i}(\hat{v}_i).
\]
To avoid ill-conditioning when $e_i$ is an extreme value, we want to choose $t$ such that 
$e_i^2t = c > 0$, a constant for $i = 1,\ldots,N$. However, this is impossible unless 
$e_1,\ldots,e_N$ are all equal, so instead we minimize $\sum_{i=1}^N (\log t - \log(ce_i^{-2}))^2$, 
where we have taken logs because $e_i$ is on the exponential scale as discussed in the previous 
section. For $c = \frac{1}{10}$, the solution is precisely our choice of $t$.

\section{Implementation}\label{solver}
We now describe the implementation details and user interface of our A2DR solver.

\paragraph{Least squares evaluation} There are three places in A2DR
that require the solution of a least squares problem. 
First, to evaluate the $F_{\rm DRS}$ projection
\[
        \Pi(v^{k+1/2}) = v^{k+1/2} - A^{\dagger}(Av^{k+1/2} - b),
\]
we solve
\[
\begin{array}{ll}
        \text{minimize} & \|Ad-(Av^{k+1/2}-b)\|_2
\end{array}
\]
with respect to $d \in \reals^n$ to obtain $d^k=A^{\dagger}(Av^{k+1/2}-b)$. 
This is accomplished in our implementation with LSQR, a conjugate gradient (CG) method \cite{lsqr}. 
Specifically, we store $A$ as a sparse matrix and call 
\texttt{scipy.sparse.linalg.lsqr} with warm start at each iteration.
LSQR has low memory requirements and converges extremely fast on well-conditioned systems, 
making it ideal for the problems we typically encounter.

Second, to compute the approximate dual variable $\lambda^k$ in 
\eqref{dual_k_ls}, we minimize $\|r_{\rm dual}^k\|_2$. We use LSQR 
with a warm start for this as well.

Finally, to solve the regularized least squares problem 
\eqref{unc_aa2sub}, we offer two options: the first is again LSQR, 
and the second is \texttt{numpy.linalg.lstsq}, an SVD-based 
least squares solver. Our implementation defaults to the second choice.
This direct method is more stable, and since $Y_k$ is a tall matrix with 
very few columns, the SVD is relatively efficient to compute at each iteration.

\paragraph{Solver interface} The A2DR solver is called with the command
~\\
\begin{center}
\texttt{result = a2dr(p\_list, A\_list, b)}
\end{center}
~\\ 
where \texttt{p\_list} is the list of proximal operators of $f_i$, 
\texttt{A\_list} is the list of $A_i$, and \texttt{b} is the vector $b$.
The lists \texttt{p\_list} and \texttt{A\_list} must be given in the same
order of $i = 1,\ldots,N$.
Each element of \texttt{p\_list} is a Python function,
which takes as input a vector $v$ and parameter $t > 0$ and outputs the proximal
operator of $f_i$ evaluated at $(v,t)$. For example, if $N=2$ with 
$f_1(x_1) = \|x_1\|_2^2$ and $f_2(x_2) = \mathcal{I}_{\mathbf{R}_+^n}(x_2)$,
~\\
\begin{center}
\texttt{p\_list = [lambda v, t: v/(1.0 + 2*t), lambda v, t: numpy.maximum(v,0)]}
\end{center}
~\\
is a valid implementation. The \texttt{result} is a Python dictionary comprised of the key/value pairs 
\texttt{x\_vals}: a list of $x_1^{k^\star+1/2},\ldots,x_N^{k^\star+1/2}$ from the iteration $k^\star$ 
with the smallest $\|r^{k^\star}\|_2$, \texttt{primal} and \texttt{dual}: arrays containing the residual 
norms $\|r_{\rm prim}^k\|_2$ and $\|r_{\rm dual}^k\|_2$, respectively, at each iteration $k$, 
\texttt{num\_iters}: the total number of iterations, and \texttt{solve\_time}: the algorithm runtime.

Arguments \texttt{A\_list} and \texttt{b} are optional, and when omitted, the solver 
recognizes the problem as \eqref{general} without the constraint $Ax = b$. 
All other hyperparameters in Algorithm \ref{A2DR_alg}, the initial point $v^0$, 
as well as the choice of whether to use preconditioning and/or AA, are also optional. By default, both preconditioning and AA are enabled.

Last but not least, the distributed execution of the iteration 
steps, including the evaluation of the proximal operators and 
componentwise summation and subtraction, is implemented with the 
\texttt{multiprocessing} package in Python.

\section{Numerical experiments}\label{experiments}
The following experiments were carried out on a Linux 
server with $64$ 8-core Intel Xeon E5-4620 / $2.20$ GHz processors 
and $503$ GB of RAM. We used the default A2DR solver parameters 
throughout. In particular, the AA max-memory $M_{\max} = 10$, 
regularization coefficient $\eta = 10^{-8}$, safeguarding constants 
$D=10^6$, $\epsilon=10^{-6}$, and $R=10$, and initial $v^0=0$. We set 
the stopping tolerances to $\epsilon_{\rm abs}=10^{-6}$ and 
$\epsilon_{\rm rel}=10^{-8}$ and limited the maximum number of iterations 
to $1000$ unless otherwise specified. All data were generated such 
that the problems are feasible and bounded, and hence convergence of the 
primal and dual residuals is guaranteed. While it is possible to 
improve convergence with additional parameter tuning, we emphasize that 
A2DR consistently outperforms DRS by a factor of three or more using the 
solver defaults. This performance gain is robust across all problem instances.

For each experiment, we plotted the residual norm $\|r^k\|_2$ at each 
iteration $k$ for both A2DR and vanilla DRS. The plots against runtime are 
very similar since the AA overhead is less than $10\%$ of the per-iteration 
cost, so we refrain from showing them here. We also compared the final 
objective value and constraint violations with the solution obtained by 
CVXPY \cite{cvxpy,cvxpy_rewriting}. In all but a few problem instances, 
the results match within $10^{-4}$. The results that differ are due to 
CVXPY's solver failure, which we discuss in more detail below.

\subsection{Nonnegative least squares}
The nonnegative least squares problem is
\BEQ\label{eq:nnls}
\begin{array}{ll}
\mbox{minimize} & \|Fz-g\|_2^2 \\
\mbox{subject to} & z \geq 0,
\end{array}
\EEQ
where $z \in \reals^q$ is the variable, and $F\in\reals^{p\times q}$ and $g\in\reals^p$ are problem data.
This problem may be rewritten in form (\ref{general}) by letting
\[
f_1(x_1) = \|Fx_1 - g\|_2^2, \quad f_2(x_2) = \mathcal{I}_{\mathbf{R}_+^n}(x_2) 
\]
for $x_1,x_2 \in \reals^q$ and enforcing the constraint $x_1 = x_2$ 
with $A_1 = I, A_2 = -I$, and $b = 0$.
The proximal operators of $f_1$ and $f_2$ are
\begin{equation}
\begin{split}
\mbox{\bf prox}_{tf_1}(v) &= \mbox{argmin}_{x_1} \left\| 
\left[\begin{array}{c} F \\ \frac{1}{\sqrt{2t}} I \end{array}\right] 
x_1 - 
\left[\begin{array}{c} g \\ \frac{1}{\sqrt{2t}} v \end{array}\right]
\right\|_2^2, \\
\mbox{\bf prox}_{tf_2}(v) &= (v)_+. 
\end{split}
\end{equation}
We evaluate $\prox_{tf_1}$ using LSQR.

\paragraph{Problem instance} Let $p=10000$ and $q=8000$. 
We took $F$ to be a sparse random matrix with $0.1\%$ nonzero entries, 
which are drawn i.i.d. (independently and identically distributed) from 
$\mathcal{N}(0,1)$, and $g$ to be a random vector from $\mathcal{N}(0,I)$.

The convergence results are shown in Figure \ref{nnls}. 
A2DR achieves $\|r^k\|_2 \leq 10^{-6}$ in under 400 iterations, 
while DRS flattens out at $\|r^k\|_2 \approx 10^{-2}$ until the maximum 
number of iterations is reached. Our algorithm's speed is a notable 
improvement over other popular solvers. We solved the same problem 
using an operator splitting quadratic program (OSQP) solver \cite{OSQP} 
and SCS, which took, respectively, 349 and 327 seconds to return a solution 
with tolerance $10^{-6}$. In contrast, A2DR converged in only 55 seconds 
and produced the smallest objective value up to a precision of $10^{-10}$.

In a second experiment, we set $p=300$ and $q=500$ and compared the 
performance under adaptive regularization, as described in \eqref{unc_aa2sub}, 
with no regularization and constant regularization. Figure \ref{nnls_reg} 
shows that adaptive regularization results in better convergence. By 1000 
iterations, the residual norm is nearly $10^{-6}$ in the adaptive case, 
while it is roughly $10^{-3}$ under the other two regularization schemes. 
Similar improvement arises in the examples below, but we have not included 
the plots for the sake of brevity.

\begin{figure}
        \centering
        \includegraphics[width=10.5cm]{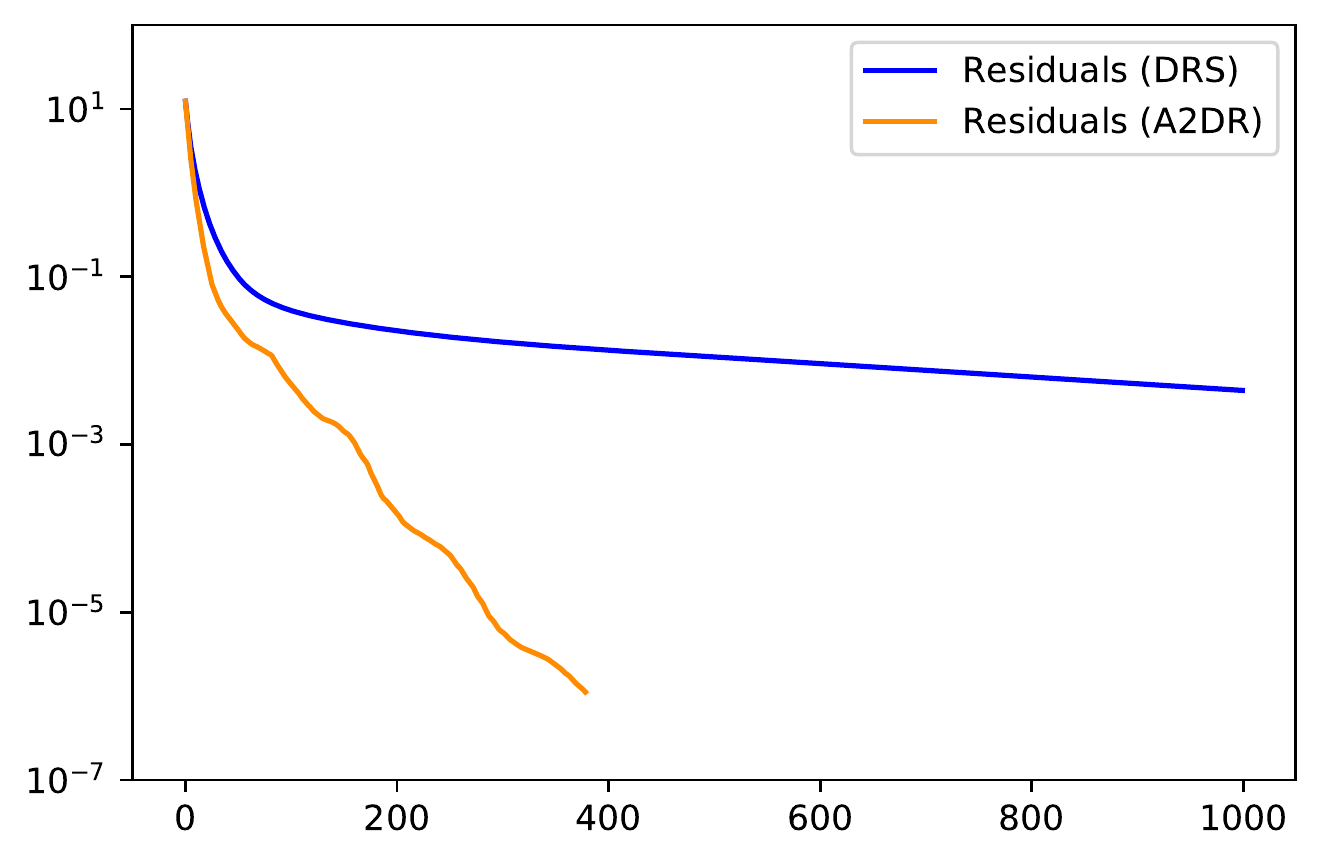}
        \caption{Nonnegative least squares: convergence of residual norms $\|r^k\|_2$.}
        \label{nnls}
\end{figure}

\begin{figure}
        \centering
        \includegraphics[width=10.5cm]{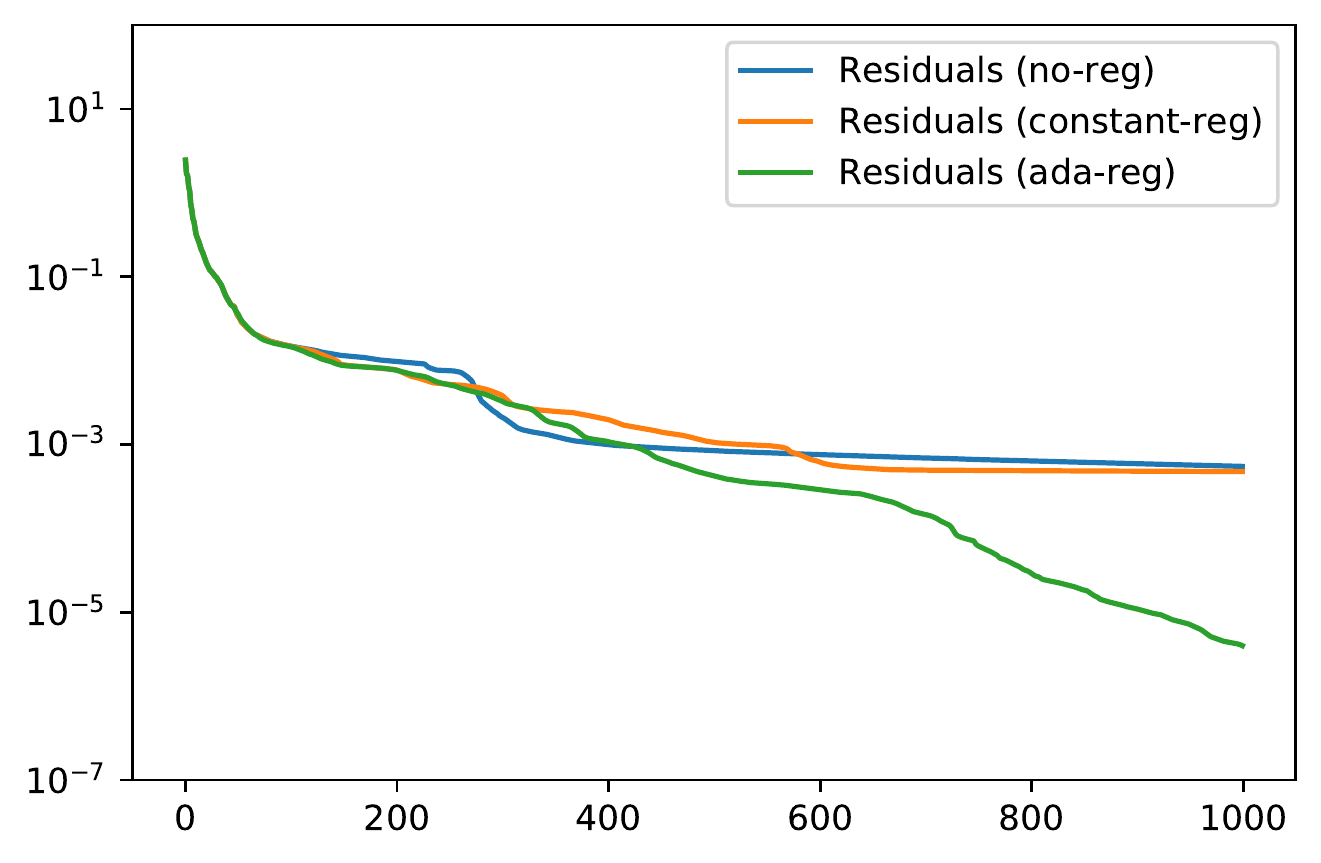}
        \caption{Nonnegative least squares: A\/$2$DR with no, constant, and adaptive regularization.}
        \label{nnls_reg}
\end{figure}

\subsection{Sparse inverse covariance estimation}
Suppose that $z_1,\dots,z_p$ are i.i.d. $N(0,\Sigma)$ 
with $\Sigma^{-1}$ known to be sparse. We can estimate the 
covariance matrix $\Sigma \in \symm_+^q$ 
by solving the optimization problem \cite{spinvcov,Banerjee:2008} 
\BEQ\label{eq:spinv}
\begin{array}{ll} 
\mbox{minimize} & -\log\det(S) + \mbox{tr}(SQ) + \alpha\|S\|_1,
\end{array}
\EEQ
where $S \in \symm^q$ (the set of symmetric matrices) is the variable, 
$Q= \frac{1}{p}\sum_{l=1}^p z_lz_l^T$ is the sample covariance, and 
$\alpha>0$ is a hyperparameter. We then take $\hat{\Sigma} = S^{-1}$ 
as an estimate of $\Sigma$. Here $\|S\|_1$ is the elementwise $\ell_1$ norm and
$\log\det$ is understood to be an extended real-valued function, \ie, 
$\log\det(S)=-\infty$ whenever $S\nsucc 0$. 

Let $x_i \in \reals^{q(q+1)/2}$ be some vectorization of $S_i 
\in \symm^q$ for $i = 1,2$. Problem \eqref{eq:spinv} can be 
represented in standard form (\ref{general}) by setting
\[
f_1(x_1) = -\log\det(S_1)+\mbox{tr}(S_1Q), \quad  f_2(x_2) = \alpha\|S_2\|_1,
\]
and $A_1=I$, $A_2=-I$, and $b=0$. 

The proximal operator of $f_1$ can be computed by combining 
the affine addition rule in \cite[section 2.2]{Proximal} with 
\cite[section 6.7.5]{Proximal}, while the proximal operator of 
$f_2$ is simply the shrinkage operator \cite[section 6.5.2]{Proximal}. 
The overall computational cost is dominated by the eigenvalue 
decomposition involved in evaluating $\prox_{tf_1}$, 
which has complexity $O(q^3)$.

\paragraph{Problem instance}
We generated $S \in \symm_{++}^q$, the set of symmetric positive definite matrices, with $q = 100$ and approximately 
$10\%$ nonzero entries. Then we calculated $Q$ using $p = 1000$ i.i.d. samples from $\mathcal{N}(0,S^{-1})$. 
Let $\alpha_{\max}=\sup_{i\neq j}|Q_{ij}|$ be the smallest $\alpha$ for which the solution of \eqref{eq:spinv}
is trivially the diagonal matrix $(\textbf{diag}(Q)+\alpha I)^{-1}$ \cite{Banerjee:2008}. We solved \eqref{eq:spinv} 
using $\alpha = 0.001\alpha_{\max}$, which produced an estimate of $S$ with $7\%$ nonzero entries. 

Figure \ref{sparse_inv_cov_est} depicts the residual norm curves. A2DR achieves $\|r^k\|_2 \leq 10^{-6}$ in 
less than 400 iterations, while DRS fails to fall below $10^{-4}$ even at 1000 iterations. 
The fluctuations in the A2DR residuals may be smoothed out by increasing the adaptive regularization coefficient $\eta$, 
but this generally leads to slower convergence.

We also ran A2DR on instances with $q = 1200$ and $q = 2000$ (vectorizations on the order of $10^6$) and 
compared its performance to SCS. In the former case, A2DR took 1 hour to converge to a tolerance of $10^{-3}$, 
while SCS took 11 hours to achieve a tolerance of $10^{-1}$ and yielded a much worse objective value. In the latter case, 
A2DR converged in 2.6 hours to a tolerance of $10^{-3}$, while SCS failed immediately with an out-of-memory error.

\begin{figure}
        \centering
        \includegraphics[width=10.5cm]{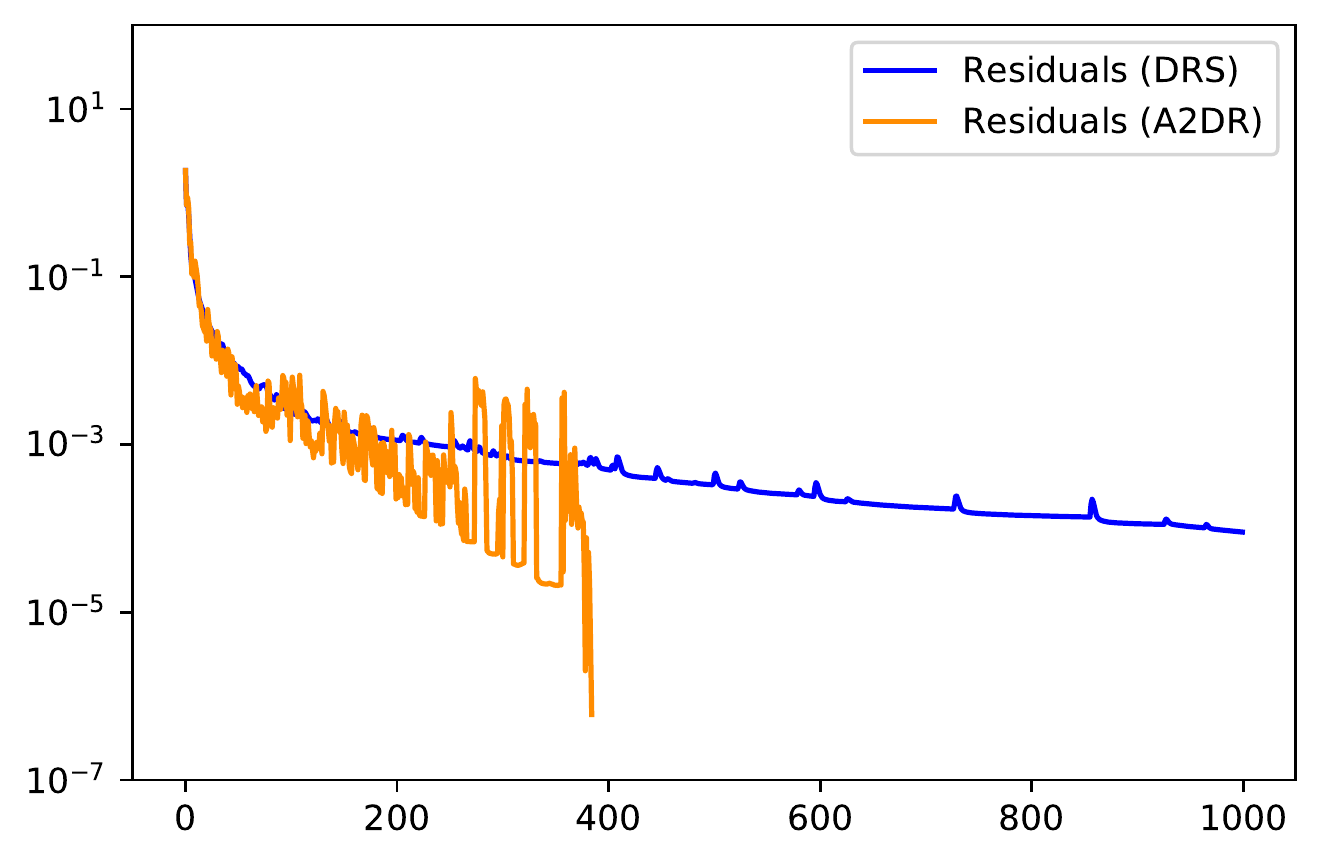}
        \caption{Sparse inverse covariance estimation: convergence of residual norms $\|r^k\|_2$.}
        \label{sparse_inv_cov_est}
\end{figure}

\subsection{\boldmath$\ell_1$ trend filtering} 
The $\ell_1$ trend filtering problem is \cite{TrendFiltering}
\BEQ\label{eq:l1trend}
\begin{array}{ll}
\mbox{minimize} & \frac{1}{2}\|y - z\|_2^2 + \alpha \|Dz\|_1,
\end{array}
\EEQ
where $z \in \reals^q$ is the variable, $y\in\reals^q$ is the problem data (\eg, time series), 
$\alpha \geq 0$ is a smoothing parameter, and $D \in \reals^{(q-2) \times q}$ is 
the second difference operator
\[
D = \left[\begin{array}{ccccccc}
1 & -2 &  1 & 0 & \ldots & 0  &0 \\
0 &  1 & -2 & 1 & \ldots & 0 & 0  \\
\vdots & \vdots & \ddots & \ddots & \ddots & \vdots& \vdots \\
0 & 0 & \ldots &1 & -2 & 1 & 0 \\
0 & 0 & \ldots & 0 & 1 & -2 & 1 
\end{array}\right].
\]
Again, we can rewrite the above problem in standard form 
\eqref{general} by letting 
\[
f_1(x_1) = \frac{1}{2}\|y - x_1\|_2^2, \quad f_2(x_2) = 
\alpha\|x_2\|_1
\]
with variables $x_1 \in \reals^q, x_2 \in \reals^{q-2}$ and 
constraint matrices $A_1 = D, A_2 = -I$, and $b = 0$. The proximal 
operator of $f_1$ is simply $\mbox{\bf prox}_{tf_1}(v) = 
\frac{ty + v}{t + 1}$, and the proximal operator of 
$f_2$ is the shrinkage operator \cite[section 6.5.2]{Proximal}. 
Since $D$ is tridiagonal, the projection $\Pi(v^{k+1/2})$ can 
be computed in $O(q)$.

\paragraph{Problem instance} We drew $y$ from $\mathcal{N}(0,I)$ with $q=10^6$ and solved \eqref{eq:l1trend} 
using $\alpha=0.01\alpha_{\max}$, where $\alpha_{\max}=\|y\|_{\infty}$ is the smallest $\alpha$ for which 
the solution is trivially zero.

The results are shown in Figure \ref{l1_trend_filter}. A2DR converges about three times faster than DRS, 
reaching a tolerance of $10^{-6}$ in 360 iterations.

\begin{figure}
        \centering
        \includegraphics[width=10.5cm]{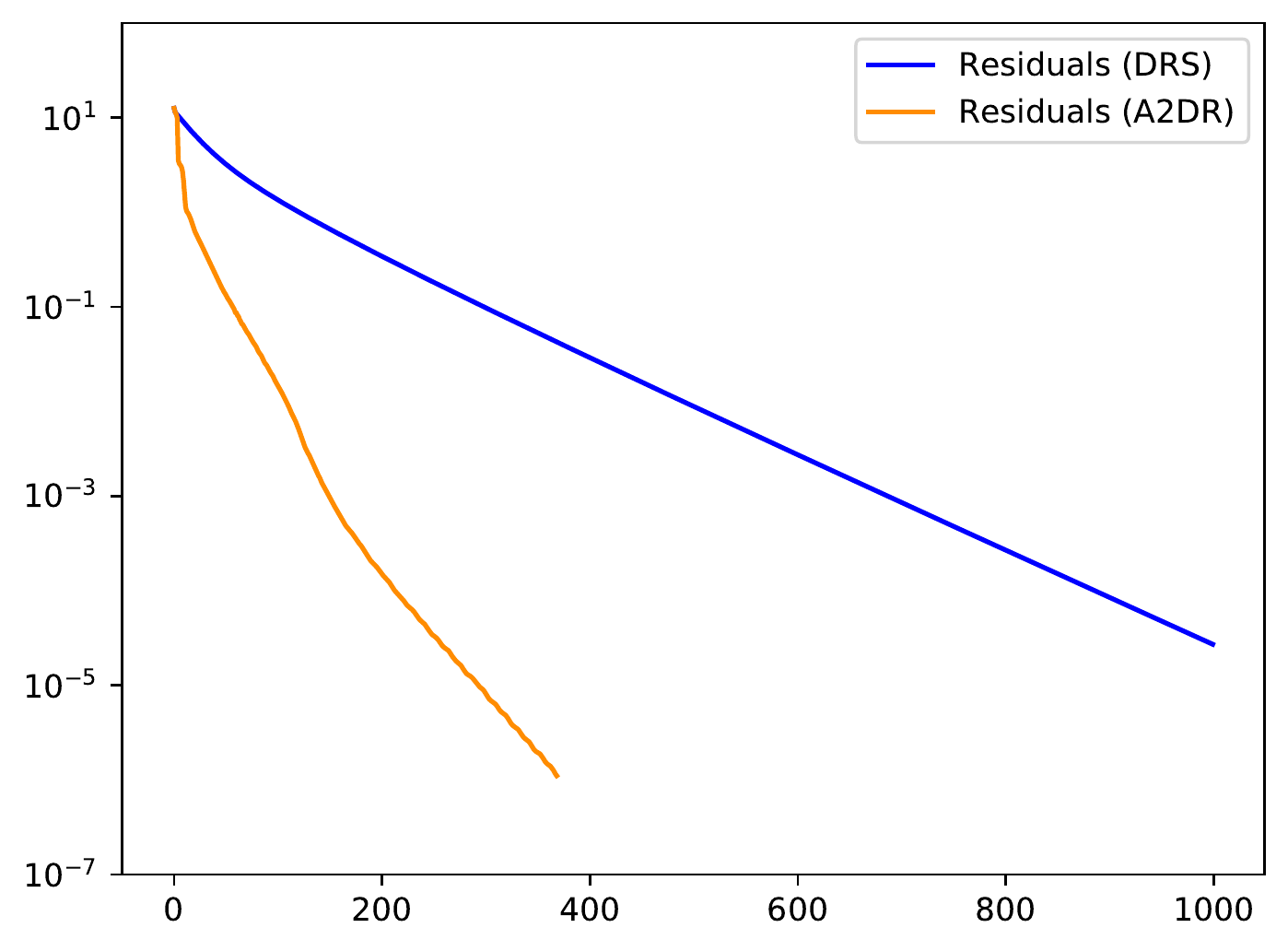}
        \caption{$\ell_1$ trend filtering: convergence of residual norms $\|r^k\|_2$.}
        \label{l1_trend_filter}
\end{figure}  

\subsection{Single commodity flow optimization} Consider a 
network with $p$ nodes and $q$ (directed) arcs described by 
an incidence matrix $B \in \reals^{p \times q}$ with
\[
B_{ij} = \left\{\begin{array}{rl}
1 & \mbox{arc $j$ enters node $i$},\\
-1 & \mbox{arc $j$ leaves node $i$},\\
0 & \mbox{otherwise}.\end{array}\right.
\]
Suppose a single commodity flows in this network. Let $z \in \reals^q$ 
denote the arc flows and $s \in \reals^p$ the node sources. We 
have the flow conservation constraint $Bz + s = 0$. This in turn
implies $\ones^Ts = 0$ since $B^T\ones = 0$ by construction. The 
total cost of traffic on the network is the sum of a flow cost, 
represented by $\psi:\reals^q \rightarrow \reals \cup \{\infty\}$, 
and a source cost, represented by $\phi:\reals^p \rightarrow \reals 
\cup \{\infty\}$. We assume that these costs are separable with 
respect to the flows and sources, \ie, 
$\psi(z) = \sum_{j=1}^q \psi_j(z_j)$ and $\phi(s) = \sum_{i=1}^p 
\phi_i(s_i)$. Our goal is to choose flow and source vectors such that the 
network cost is minimized:
\BEQ
\begin{array}{ll}
\mbox{minimize} & \psi(z) + \phi(s) \\
\mbox{subject to} & Bz + s = 0
\end{array}
\EEQ
with respect to $z$ and $s$. 

We consider a special case modeled on the DC power flow problem in power engineering \cite{Dyn_Energy}. 
The flow costs are quadratic with a capacity constraint:
\[
\psi_j(z_j) = \begin{cases}
c_jz_j^2, & |z_j| \leq z_j^{\max}, \\
+\infty & \mbox{otherwise.}
\end{cases}
\]
The source costs are determined by the node type, 
which can fall into one of three categories:
\begin{enumerate}
\item \emph{Transfer/way-point nodes} fixed at $s_i = 0$, \ie, 
$\phi_i(s_i)=\mathcal{I}_{\{0\}}(s_i)$.
\item \emph{Sink nodes} fixed at $s_i = L_i$ (for $L_i < 0$), \ie, 
$\phi_i(s_i)=\mathcal{I}_{\{L_i\}}(s_i)$.
\item \emph{Source nodes} with cost
\[
\phi_i(s_i) = \begin{cases}
d_is_i^2, & 0 \leq s_i \leq s_i^{\max}, \\
+\infty & \mbox{otherwise}.
\end{cases}     
\]
\end{enumerate}
The vectors $c \in \reals_+^q,d \in \reals_+^p,z^{\max} 
\in \reals^q$, and $s^{\max} \in \reals^p$ are constants.

This problem may be restated as \eqref{general} with $x_1 \in \reals^q, x_2 \in \reals^p, f_1(x_1) = 
\psi(x_1),\break f_2(x_2) = \phi(x_2), A_1 = B, A_2 = I$, and $b = 0$. 
Since costs are separable, the proximal operators can be 
calculated elementwise as
\begin{equation}
\begin{split}
(\mbox{\bf prox}_{tf_1}(v))_j &= \Pi_{[-z_j^{\max},z_j^{\max}]}
\left(\frac{v_j}{2 tc_j + 1}\right), \quad j = 1,\ldots,q, \\
(\mbox{\bf prox}_{tf_2}(w))_i &= \Pi_{[0,s_i^{\max}]}
\left(\frac{w_i}{2 td_i + 1}\right), \quad i 
\mbox{ is a source node.}
\end{split}
\end{equation}
Here $\Pi_{\mathcal{C}}$ denotes the projection onto the set $\mathcal{C}$. Notice that in evaluating the 
proximal operator, we implicitly solve a linear system related to $L = BB^T$, which is the Laplacian 
associated with the network.

\paragraph{Problem instance} 
We set $p=4000$ and $q=7000$ and generated the incidence matrix as follows. 
Let $\tilde{B}\in\reals^{p\times (q-p+1)}$, where each column $j$ is zero except 
for two entries $\tilde{B}_{ij} = 1$ and $\tilde{B}_{i'j} = -1$, whose positions 
are chosen uniformly at random. Define $\hat{B}\in\reals^{p\times (p-1)}$ 
with $\hat{B}_{ii}=1$ and $\hat{B}_{(i+1)i}=-1$ for $i=1,\dots,p-1$. 
The final incidence matrix is $B=[\tilde{B}~\hat{B}]$. 

To construct the source vector, we first drew $\breve{s} \in \reals^p$ i.i.d.\ from $\mathcal{N}(0,I)$ and defined
\[
        \tilde{s}_i = \begin{cases} 
                0, & i = 1,\ldots,\lfloor\frac{p}{3}\rfloor, \\
                -|\breve{s}_{i}|, & i=\lfloor \frac{p}{3}\rfloor+1,\dots,\lfloor\frac{2p}{3}\rfloor, \\
                \sum_{l=\lfloor p/3\rfloor+1}^{\lfloor 2p/3\rfloor}|\breve{s}_{l}|
                /(p-\lfloor \frac{2p}{3}\rfloor), & i=\lfloor \frac{2p}{3}\rfloor+1,\dots,p.
        \end{cases}
\]
We took the first $\lfloor \frac{p}{3}\rfloor$ entries to be the transfer nodes, the second 
$\lfloor \frac{2p}{3}\rfloor-\lfloor \frac{p}{3}\rfloor$ entries to be the sink nodes with $L_i=\tilde{s}_i$, 
and the last $p-\lfloor \frac{2p}{3}\rfloor$ entries to be the source nodes, where
\[
        s_i^{\max} = \begin{cases}
                \tilde{s}_i+0.001, & i=\lfloor \frac{2p}{3}\rfloor,\dots,\lfloor \frac{5p}{6}\rfloor, \\
                2(\tilde{s}_i+0.001), & i=\lfloor \frac{5p}{6}\rfloor,\dots,p.
        \end{cases}
\]
To get the flow bounds, we solved $B\tilde{x}=-\tilde{s}$ for $\tilde{x}$, and let 
\[
        x_j^{\max} = \begin{cases}
                |\tilde{x}_j|+0.001, & j=1,\dots,\lfloor \frac{q}{2}\rfloor, \\
                2(|\tilde{x}_j|+0.001), & j=\lfloor \frac{q}{2}\rfloor+1,\dots,q.
        \end{cases}
\]
Finally, the entries of $c$ and $d$ were drawn i.i.d. from $\mbox{Uniform}(0,1)$.

Figure \ref{single_com_flow} depicts the results of our experiment. 
A2DR converges to a tolerance of $10^{-6}$ in less than 1200 iterations, while DRS remains above $10^{-4}$ 
even once the maximum iterations of 2000 is reached. For this problem, we also attempted to find a solution using SCS, 
but the solver failed to converge to its default tolerance of $10^{-5}$ in 5000 iterations, 
finishing with a linear constraint violation of $\|Bz + s\|_2 > 0.3$. In contrast, A2DR's final result yields 
$\|Bz + s\|_2 \approx 10^{-6}$.

 \begin{figure}
        \centering
        \includegraphics[width=10.5cm]{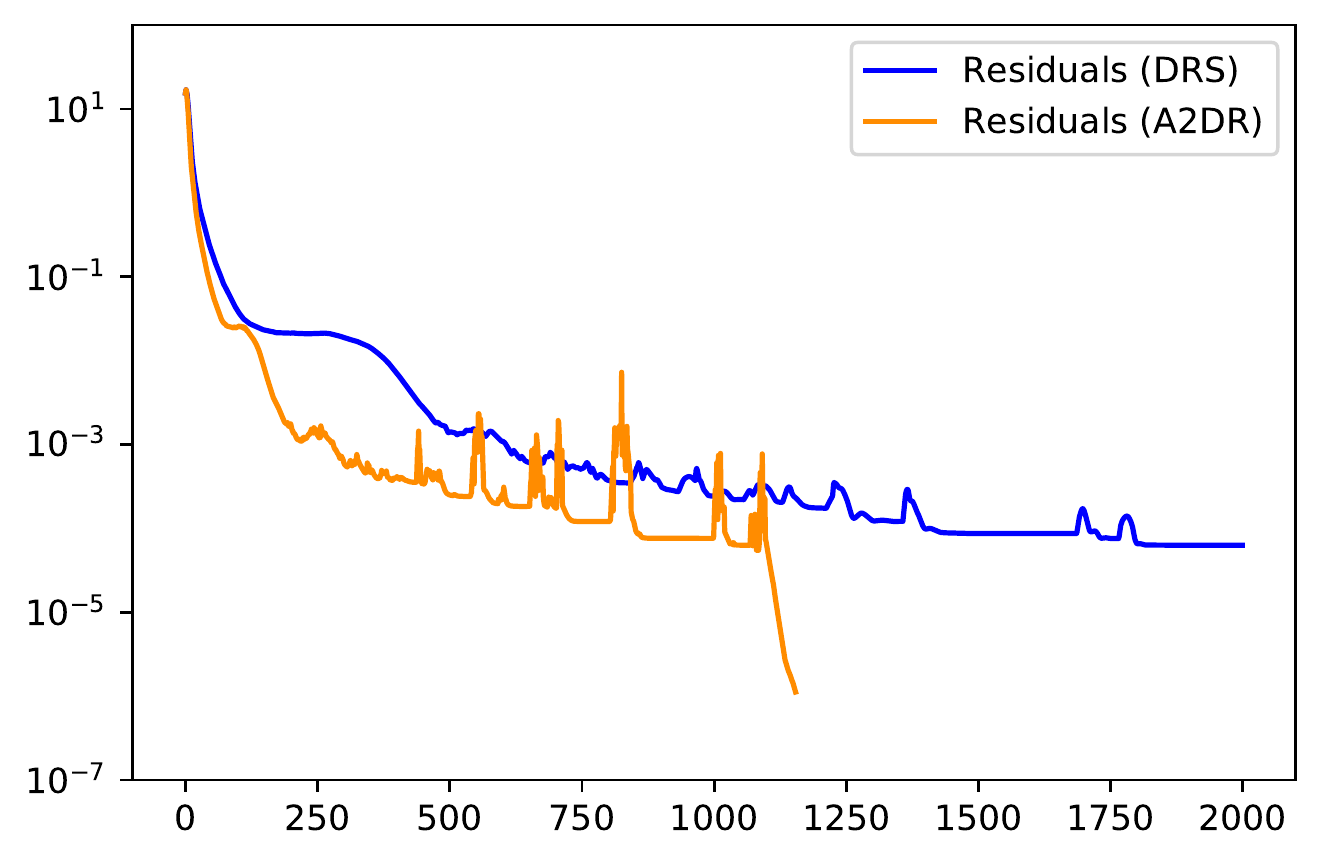}
        \caption{Single commodity flow: convergence of residual norms $\|r^k\|_2$.}
        \label{single_com_flow}
\end{figure}

\subsection{Optimal control} We are interested in the following
finite-horizon optimal control problem:
\BEQ\label{eq:opt_ctrl}
\begin{array}{ll}
\mbox{minimize} & \sum_{l=1}^L \phi_l(z_l, u_l) \\
\mbox{subject to} & z_{l+1} = F_lz_l + G_lu_l + h_l, 
\quad l = 1,\ldots,L-1, \\
& z_1 = z_{\text{init}}, \quad z_L = z_{\text{term}}
\end{array}
\EEQ
with state variables $z_l \in \reals^q$, control variables 
$u_l \in \reals^p$, and cost functions $\phi_l:\reals^q \times 
\reals^p \rightarrow \reals \cup \{\infty\}$. The data consist 
of an initial state $z_{\text{init}} \in \reals^q$, a terminal state
$z_{\text{term}}\in\reals^q$, and dynamics 
matrices $F_l \in \reals^{q \times q}, G_l \in \reals^{q \times p}$, 
and $h_l \in \reals^q$ for $l = 1,\ldots,L-1$. 
Let $z = (z_1,\ldots,z_L) \in \reals^{Lq}$ and 
$u = (u_1,\ldots,u_L) \in \reals^{Lp}$. If we define
\[
\tilde F = \left[\begin{array}{ccccc}
I &    0 & \ldots &        0 &   0 \\
-F_1 &    I & \ldots &        0 &   0 \\
0 & -F_2 & \ldots &        0 &   0 \\
\vdots & \vdots & \ddots & \vdots & \vdots \\
0 &    0 & \ldots & -F_{L-1} & I \\
0 &    0 & \ldots &        0 & I
\end{array}\right], \quad 
\tilde G = \left[\begin{array}{ccccc}
0 &    0 & \ldots &        0 & 0 \\
-G_1 &    0 & \ldots &        0 & 0 \\
0 & -G_2 & \ldots &        0 & 0 \\
\vdots & \vdots & \ddots & \vdots & \vdots \\
0 &    0 & \ldots & -G_{L-1} & 0 \\
0 &    0 & \ldots &        0 & 0
\end{array}\right],
\]
and $\tilde h = (z_{\text{init}}, h_1, \ldots, h_{L-1}, z_{\text{term}})$, 
then the constraints can be written compactly as 
$\tilde Fz + \tilde Gu = \tilde h$.

We focus on a time-invariant linear quadratic version of 
(\ref{eq:opt_ctrl}) with $F_l = F, G_l = G, h_l = 0$, 
and
\[
\phi_l(z_l,u_l) = \|z_l\|_2^2 + \|u_l\|_2^2 + 
\mathcal{I}_{\{u\,:\,\|u\|_{\infty} \leq 1\}}(u_l), 
\quad l = 1,\ldots,L.
\]
This problem is equivalent to \eqref{general} with $x_1 \in \reals^{Lq}, x_2 \in \reals^{Lp}$,
\[
f_1(x_1) = \|x_1\|_2^2, \quad f_2(x_2) = \|x_2\|_2^2 + 
\mathcal{I}_{\{u\,:\,\|u\|_{\infty} \leq 1\}}(x_2),
\]
and constraint matrices $A_1 = \tilde F, A_2 = \tilde G$, 
and $b = \tilde h$. The proximal operators of $f_i$ have 
closed forms $\mbox{\bf prox}_{tf_1}(v) = 
\frac{v}{2t+1}$ and $\mbox{\bf prox}_{tf_2}(w) 
= \Pi_{[-1,1]}\big(\frac{w}{2t+1}\big)$.

\paragraph{Problem instance}
We set $p=80, q=150$, and $L=20$ and drew the entries of $F, G, h$, and $z_{\rm init}$ i.i.d. from $\mathcal{N}(0,1)$. 
The matrix $F$ was scaled by its spectral radius so its largest eigenvalue has magnitude one. 
To determine $z_{\rm term}$, we drew $\hat u_l \in \reals^p$ i.i.d. from $\mathcal{N}(0,I)$, normalized to get 
$\tilde u_l = \hat u_l/\|\hat u_l\|_{\infty}$, and computed $\tilde z_{l+1} = F\tilde z_l + G\tilde u_l + h$ 
for $l = 1,\ldots,L-1$ starting from $\tilde z_1 = z_{\rm init}$. We then chose the terminal state to be 
$z_{\rm term} = \tilde z_L$.

Figure \ref{opt_cont} depicts the residual curves for problem (\ref{eq:opt_ctrl}). DRS requires over five times 
as many iterations to converge as A2DR, which reaches a tolerance of $10^{-6}$ in just under 100 iterations. 
For comparison, we solved the same problem in CVXPY with OSQP and SCS and found that neither solver converged to its 
default tolerance ($10^{-4}$ and  $10^{-5}$, respectively) by its maximum number of iterations. 
Indeed, OSQP returned a solver error, while SCS terminated with a linear constraint violation of 
$\|\tilde Fz + \tilde Gu - \tilde h\|_2 > 0.9$. A2DR's final constraint violation is only about $10^{-6}$.

 \begin{figure}
        \centering
        \includegraphics[width=10.5cm]{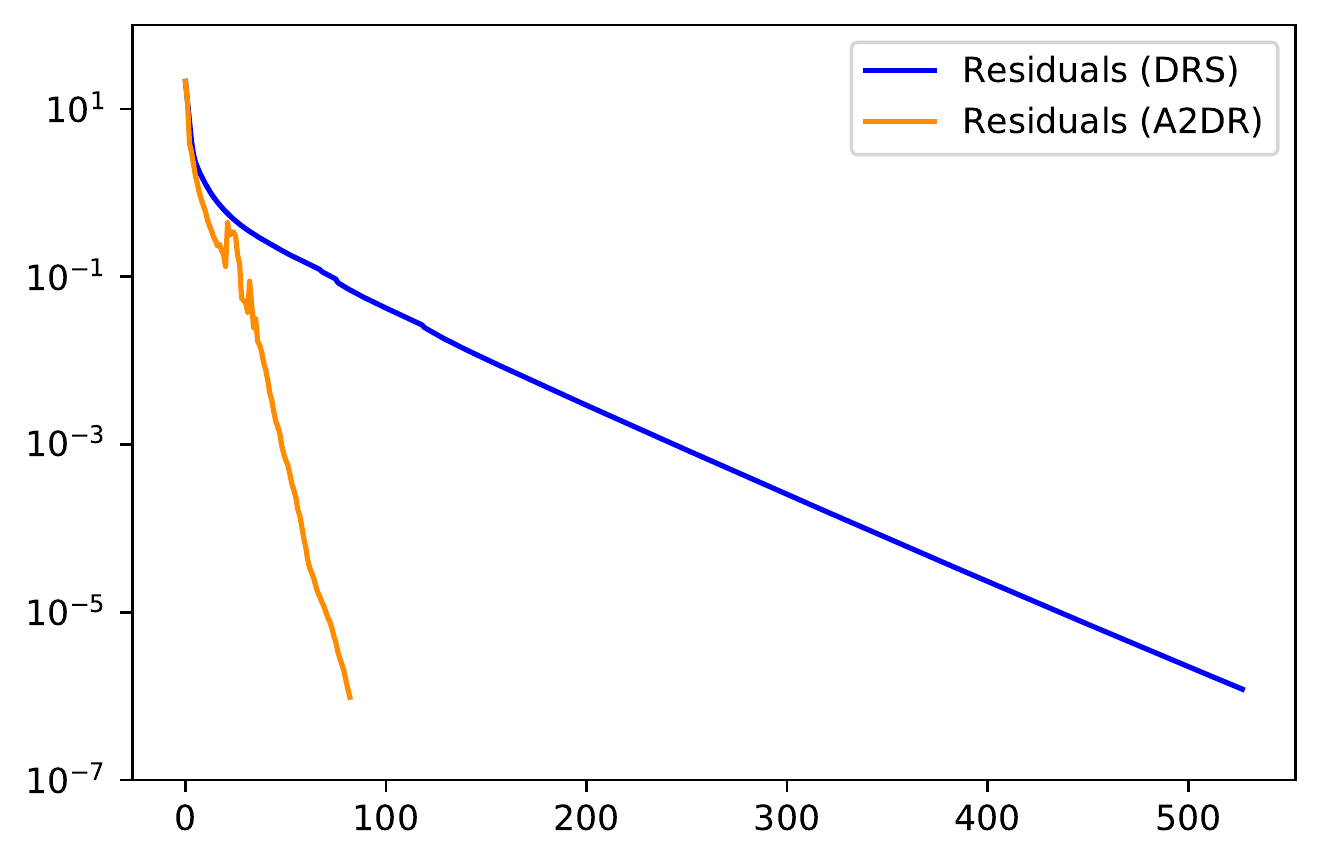}
        \caption{Optimal control: convergence of residual norms $\|r^k\|_2$.}
        \label{opt_cont}
\end{figure}

\subsection{Coupled quadratic program} We consider a quadratic 
program in which $L$ variable blocks are coupled through a set 
of $s$ linear constraints, represented as 
\BEQ\label{ms_oc}
\begin{array}{ll}
\mbox{minimize} & \sum_{l=1}^L z_l^TQ_lz_l+c_l^Tz_l\\
\mbox{subject to} & F_lz_l\leq d_l,\quad l=1,\dots,L,\\
& \sum_{l=1}^LG_lz_l=h
\end{array}
\EEQ
with respect to $z=(z_1,\dots,z_L)$, where $z_l \in \reals^{q_l}, 
Q_l\in \symm_{+}^{q_l}, c_l\in\reals^{q_l}, F_l \in \reals^{p_l 
\times q_l}, d_l \in \reals^{p_l}, G_l \in \reals^{s \times q_l}$, 
and $h \in \reals^s$ for $l = 1,\ldots,L$.

We can rewrite (\ref{ms_oc}) in standard form with $N=L$, $x=z$,
\[
f_i(x_i)=x_i^TQ_ix_i+c_i^Tx_i+\mathcal{I}_{\{x\,:\,F_ix\leq d_i\}}(x_i), 
\quad i = 1,\ldots,L,
\]
$A=[G_1~\cdots~G_L]$, and $b = h$. The proximal operator $\prox_{tf_i}(v_i)$ 
is evaluated by solving 
\BEQ\label{ms_oc_std}
\begin{array}{ll}
\mbox{minimize} & x_i^T\left(Q_i+\frac{1}{2t}I\right)
x_i +(c_i-\frac{1}{t}v_i)^Tx_i\\
\mbox{subject to} & F_ix_i\leq d_i\\
\end{array}
\EEQ
with respect to $x_i \in \reals^{q_i}$.

\paragraph{Problem instance} Let $L=8$, $s=50$, $q_l=300$, and $p_l=200$ for $l=1,\ldots,L$. We generated the entries of 
$c_l \in \reals^{q_l}, F_l \in \reals^{p_l \times q_l}, G_l \in \reals^{s \times q_l}, \tilde z_l \in \reals^{q_l}$, 
and $H_l \in \reals^{q_l \times q_l}$ i.i.d. from $\mathcal{N}(0,1)$. We then formed $d_l = F_l\tilde z_l + 0.1, Q_l = H_l^TH_l$, 
and $h = \sum_{l=1}^L G_l\tilde z_l$. To evaluate the proximal operators, we constructed problem \eqref{ms_oc_std} 
in CVXPY and solved it using OSQP with the default tolerance.

The results of our experiment are shown in Figure \ref{coupled_qp}. A2DR produces an over ten-fold speedup, 
converging to the desired tolerance of $10^{-6}$ in only $60$ iterations.

\begin{figure}
        \centering
        \includegraphics[width=10.5cm]{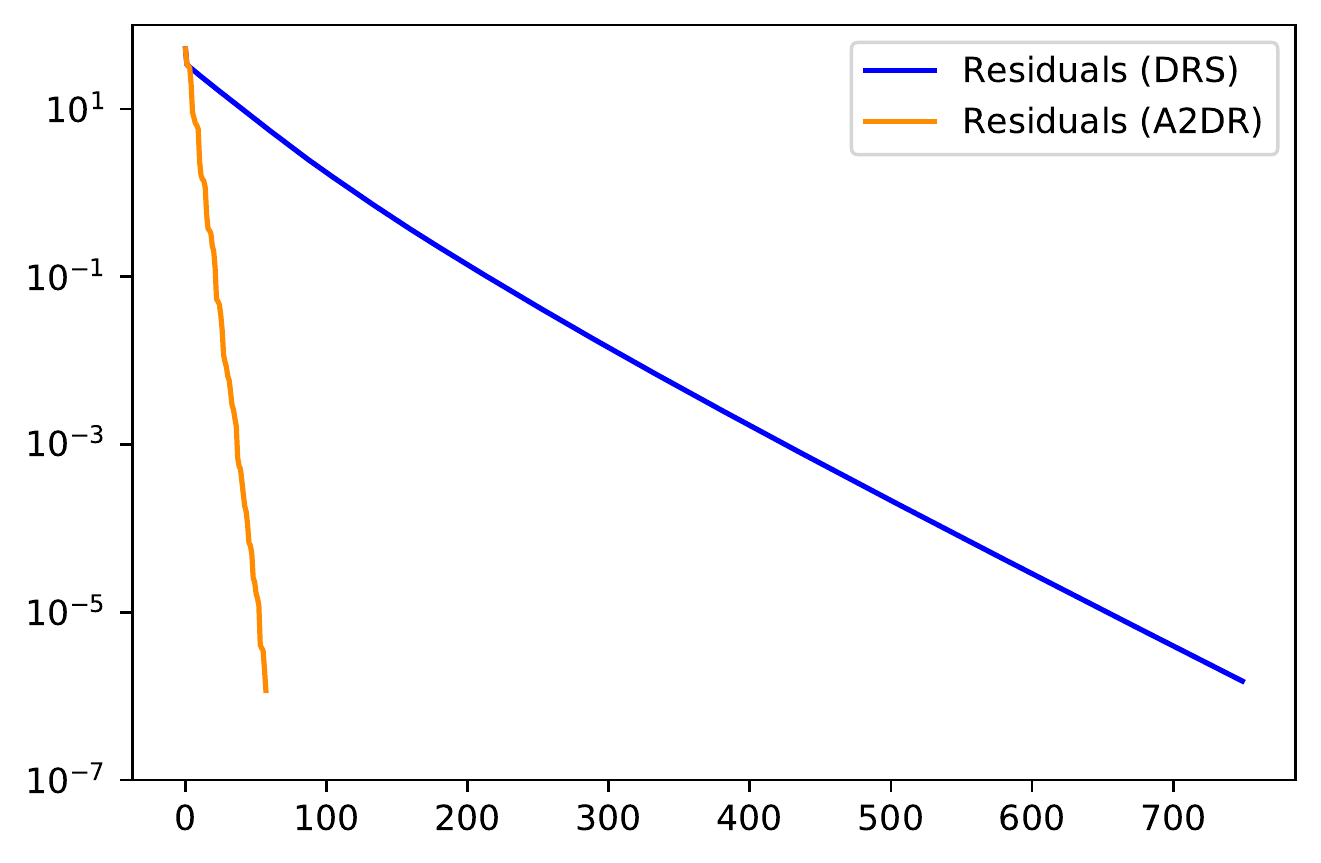}
        \caption{Coupled quadratic program: convergence of residual norms $\|r^k\|_2$.}
        \label{coupled_qp}
\end{figure}

\subsection{Multitask regularized logistic regression} Consider
the following multi-task regression problem:
\BEQ\label{ms_logreg}
\begin{array}{ll}
\mbox{minimize} & \phi(W\theta,Y) + r(\theta)
\end{array}
\EEQ
with variable $\theta=[\theta_1 \cdots \theta_L] 
\in\reals^{s\times L}$. Here $\phi:\reals^{p \times L} 
\times \reals^{p \times L} \rightarrow \reals$ is the loss function, 
$r:\reals^{s \times L} \rightarrow \reals$ is the regularizer, 
$W\in\reals^{p\times s}$ is the feature matrix shared across the 
$L$ tasks, and $Y=[y_1 \cdots y_L]\in\reals^{p\times L}$ contains 
the $p$ class labels for each task $l=1,\ldots,L$. 

We focus on the binary classification problem, so that all entries 
of $Y$ are $\pm 1$. Accordingly, we take our loss function to be 
the logistic loss summed over samples and tasks,
\[
\phi(Z,Y)= \sum_{l=1}^L\sum_{i=1}^p\log\left(1+\exp(-Y_{il}Z_{il})\right),
\]
where $Z \in \reals^{p \times L}$,
and our regularizer to be a linear combination of the group lasso 
penalty \cite{group_lasso} and the nuclear norm,
\[
r(\theta)=\alpha\|\theta\|_{2,1}+\beta\|\theta\|_*,
\]
where $\|\theta\|_{2,1}=\sum_{l=1}^L\|\theta_l\|_2$ and 
$\alpha>0,~\beta>0$ are regularization parameters.

Problem (\ref{ms_logreg}) can be converted to standard form 
\eqref{general} by letting 
\[
f_1(Z) = \phi(Z,Y), \quad f_2(\theta) = \alpha\|\theta\|_{2,1}, 
\quad f_3(\tilde \theta) = \beta\|\tilde \theta\|_*,
\]
\[
A = \left[\begin{array}{cccc}
I & -W & 0 \\
0 & I & -I
\end{array}\right],
\quad x = \left[\begin{array}{c} Z \\ \theta \\ \tilde\theta 
\end{array}\right],
\quad b = 0.
\]
The proximal operator of $f_1$ can be evaluated efficiently via
Newton type methods applied to each component in parallel \cite{logistic_prox}, 
while the proximal operators of the regularization terms have closed-form 
expressions \cite[sections 6.5.4 and 6.7.3]{Proximal}.

\paragraph{Problem instance} We let $p=300$, $s=500$, $L=10$, and $\alpha=\beta=0.1$. 
The entries of $W\in\reals^{p\times s}$ and $\theta^\star\in \reals^{s\times L}$ were drawn 
i.i.d. from $\mathcal{N}(0,1)$. We calculated $Y=\textbf{sign}(W\theta^\star)$, where the signum 
function is applied elementwise with the convention $\textbf{sign}(0)=-1$. To evaluate $\prox_{tf_1}$, 
we used the Newton-CG method from \texttt{scipy.optimize.minimize}, warm starting each iteration with 
the output from the previous iteration. (Further performance improvements may be achieved by 
implementing Newton's method with unit step size and initial point zero for each component in 
parallel \cite{logistic_prox}.)

Figure \ref{multitask_reglog} shows the residual plots for A2DR and DRS. The A2DR curve exhibits a 
steep drop in the first few steps and continues falling until convergence at 500 iterations. 
In contrast, the DRS residual norms never make it below a tolerance of $10^{-2}$.

\begin{figure}
        \centering
        \includegraphics[width=10.5cm]{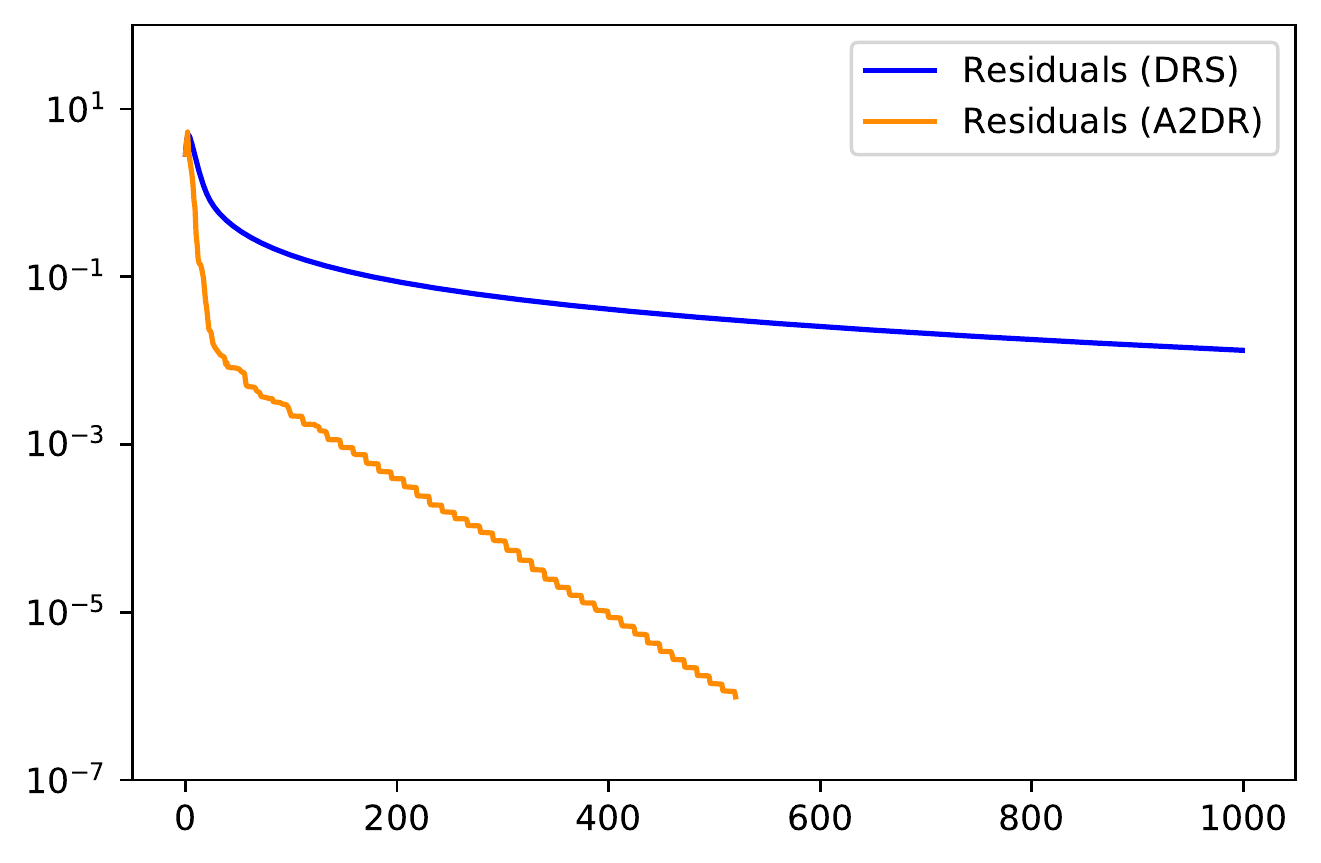}
        \caption{Multitask regularized logistic regression: convergence of residual norms $\|r^k\|_2$.}
        \label{multitask_reglog}
\end{figure}

\section{Conclusions}
\label{conclusions}
We have presented an algorithm for solving linearly constrained convex optimization problems, 
where the objective function is only accessible via its proximal operator. Our algorithm is an 
application of type-II Anderson acceleration to Douglas--Rachford splitting (A2DR). 
Under relatively mild conditions, we prove that A2DR either converges to a global optimum or 
provides a certificate of infeasibility/unboundedness. 
Moreover, when the objective is block separable, its steps partially decouple so that they 
may be computed in parallel, enabling fast distributed implementations. 
We provide one such Python implementation at \url{https://github.com/cvxgrp/a2dr}. 
Using only the default parameters, we show that our solver achieves rapid convergence on a 
wide range of problems, making it a robust choice for general large-scale convex optimization.

In the future, we plan to release a user-friendly interface, which automatically reduces a problem 
to the standard form \eqref{general} input of the A2DR solver, similar to the Epsilon 
system \cite{Epsilon}. This will allow us to integrate A2DR into a high-level 
domain specific language for convex optimization. We also intend to expand the library of proximal 
operators. As problems grow larger, we aim to support more parallel computing architectures, 
allowing users to leverage GPU acceleration and high-performance clusters for distributed optimization.

\section*{Acknowledgment}
The authors would like to thank Brendan O'Donoghue for his advice on preconditioning and 
his inspirational work developing solvers with Anderson acceleration, pioneered by SCS 2.0. 

\newpage
\appendix
\begin{center}
\Large \textbf{Supplementary Materials}\vspace{0.75em}
\end{center}
\addcontentsline{toc}{section}{Appendix}
\renewcommand{\thesubsection}{\Alph{subsection}}

\noindent In this supplementary material, we provide the proofs for the theorems in the main text. 

\section{Preliminaries}
We begin with the following lemma, which establishes the 
connection between residuals of the DRS fixed-point mapping
and the primal/dual residuals of the original problem 
\eqref{general}.

\begin{lemma}\label{fp_prob_res}
	Suppose that $\liminf_{j\rightarrow\infty}\|v^j-F_{\rm DRS}(v^j)\|_2
	\leq \epsilon$ for some $\epsilon\geq 0$. Then 
	\begin{equation}
	\liminf_{j\rightarrow\infty}\|r_{\rm prim}^j\|_2\leq 
	\|A\|_2 \epsilon, \quad \liminf_{j\rightarrow\infty}
	\|r_{\rm dual}^j\|_2\leq \frac{1}{t}\epsilon.
	\end{equation}
\end{lemma}

\begin{proof} 
	By expanding $F_{\rm DRS}$, and in particular line 6 of 
	Algorithm \ref{DRS_alg}, we see that 
	\[
	\liminf_{j\rightarrow\infty}\|x^{j+1/2}-x^{j+1}\|_2=
	\liminf_{j\rightarrow\infty}\|v^j-v_{\rm DRS}^{j+1}\|_2
	\leq \epsilon.
	\]
	
	Since $Ax^{j+1}=b$ by the projection step in $F_{\rm DRS}$, 
	we have 
	\[
	r_{\rm prim}^j=Ax^{j+1/2}-b=A(x^{j+1/2}-x^{j+1}),
	\] 
	which implies that
	\[
	\liminf_{j\rightarrow\infty}\|r_{\rm prim}^j\|_2 \leq 
	\|A\|_2\liminf_{j\rightarrow\infty}\|x^{j+1/2}-x^{j+1}\|_2
	\leq \|A\|_2\epsilon,
	\]
	and hence $\liminf_{j\rightarrow\infty}\|r_{\rm prim}^j\|_2
	\leq \|A\|_2\epsilon$. 
	
	On the other hand, the optimality conditions from lines 3 
	and 5 of Algorithm \ref{DRS_alg} give us
	\[
	\frac{1}{t}(x^{j+1/2}-v^j)+g^j=0,\quad x^{j+1}=v^{j+1/2}-A^T
	\tilde{\lambda}^j,
	\]
	for some $g^j\in\partial f(x^{j+1/2})$ and $\tilde{\lambda}^j
	=(AA^T)^{\dagger}(Av^{j+1/2}-b)$. Thus, 
	\BEQ
	\begin{split}
		g^j&=\frac{1}{t}(v^j-x^{j+1/2})\\
		&=\frac{1}{t}(v^{j+1/2}-x^{j+1})+\frac{1}{t}(v^j-v^{j+1/2})+
		\frac{1}{t}(x^{j+1}-x^{j+1/2})\\
		&=\frac{1}{t} A^T\tilde{\lambda}^j+2\frac{1}{t}(v^j-x^{j+1/2})+
		\frac{1}{t}(x^{j+1}-x^{j+1/2})\\
		&=\frac{1}{t} A^T\tilde{\lambda}^j+2g^j+
		\frac{1}{t}(x^{j+1}-x^{j+1/2}),
	\end{split}
	\EEQ
	where we have used line 4 of Algorithm \ref{DRS_alg} in the third equality. Rearranging
	terms yields $g^j=A^T(-\frac{1}{t}\tilde{\lambda}^j)+
	\frac{1}{t}(x^{j+1/2}-x^{j+1})$.

	Finally, since we compute $r_{\rm dual}^j=g^j+A^T\lambda^j$
	using $\lambda^j\in \argmin_{\lambda}\|g^j+A^T\lambda\|_2$ 
	(c.f. residuals and dual variables in \S \ref{DRS}),
	\[
	\liminf_{j\rightarrow\infty}\|r_{\rm dual}^j\|_2\leq 
	\liminf_{j\rightarrow\infty}\|g^j+A^T\bar{\lambda}^j\|_2
	=\frac{1}{t}\liminf_{j\rightarrow\infty}\|x^{j+1/2}-x^{j+1}\|_2\leq
	\frac{1}{t}\epsilon,
	\]
	where $\bar{\lambda}^j=\frac{1}{t}\tilde{\lambda}^j$. This completes
	our proof.
\end{proof}

\begin{remark}\label{rmk_app}
	When $\epsilon=0$, Lemma \ref{fp_prob_res} implies that
	\[
	\liminf_{j\rightarrow\infty}\|r_{\rm prim}^j\|_2=
	\liminf_{j\rightarrow\infty}\|r_{\rm dual}^j\|_2=0.
	\]
	Furthermore, notice that we could have calculated 
	$r_{\rm dual}^j$ using 
	\[
	\lambda^j = \bar{\lambda}^j = \frac{1}{t}(AA^T)^{\dagger}(Av^{k+1/2}-b),
	\]
	and the results would still hold.
\end{remark}

\section{Proof of Theorems \ref{global_conv_well} and 
	\ref{global_conv_path}}

We now prove the convergence results in the error-free 
setting.
Define the infimal displacement vector of $F_{\rm DRS}$ as 
$\delta v^\star = \Pi_{\overline{\range(I-F_{\rm DRS})}}(0)$.  
It follows directly that $\|\delta v^\star\|_2=
\inf_{v\in\mathbf{R}^n}~\|v-F_{\rm DRS}(v)\|_2$. We will later 
show that in A2DR, $\lim_{k\rightarrow\infty}v^k-v^{k+1}
=\delta v^\star$. In particular, Theorem \ref{global_conv_path} gives us 
$\delta v=\delta v^\star$.

We begin by showing that $\delta v^\star=0$ if and only if problem
\eqref{general} is solvable. To see this, first notice that by 
\cite[Corollary 6.5]{range_DRS},  
\[
\delta v^\star=\argmin\nolimits_{z\in \mathcal{Z}} \|z\|_2,
\]
where 
\[
\mathcal{Z}=\overline{\dom f -\dom g}\cap 
t(\overline{\dom f^* +\dom g^*}),\quad g(x)
=\mathcal{I}_{\{v\,:\,Av=b\}}(x).
\]
Since $\dom g=\{x\,:\,Ax=b\}$ and 
$\dom g^*=\range(A^T)=-\range(A^T)$, the problem is 
solvable if and only if 
\[
\dist(\dom f, \dom g)=\dist(\dom f^*, -\dom g^*)=0,
\] which holds if and only if $0\in\overline{\dom f 
	-\dom g}$ and $0\in \overline{\dom f^*+\dom g^*}$, \ie, 
$\delta v^\star=0$. 

Below we denote the initial iteration counts 
for accepting AA candidates as $k_i$ (\ie, when 
$I_{\rm safeguard}$ is \texttt{True} or $R_{\rm AA}\geq R$, 
and the check in Algorithm \ref{A2DR_alg}, line 14 passes), and the 
iteration counts 
for accepting DRS candidates as $l_i$. Notice that for each 
iteration $k$, either $k=k_i+K$ for some $i$ and 
$0\leq K\leq R-1$, or $k=l_i$ for some $i$. 
\begin{itemize}
	\item \textbf{Case (i) [Theorem \ref{global_conv_well}, 
		\eqref{liminf}]} 
	
	First, suppose that problem \eqref{general} is solvable. 
	Then, $\delta v^\star=0$. By Lemma \ref{fp_prob_res}, 
	to prove \eqref{liminf}, it suffices to prove that 
	$\liminf_{k\rightarrow\infty}\|g^k\|_2=0$. If the set of $k_i$ is infinite, \ie, 
	the AA candidate 
	is adopted an infinite number of times, then 
	\[
	0\leq \liminf_{k\rightarrow\infty}\|g^k\|_2\leq 
	\liminf_{i\rightarrow\infty}\|g^{k_i}\|_2\leq D\|g^0\|_2
	\lim_{i\rightarrow\infty}(i+1)^{-(1+\epsilon)}=0.
	\]
	Here we used the fact that $n_{AA}/R=i$ in iteration $k_i$. 
	
	On the other hand, if the set of $k_i$ is finite, Algorithm \ref{A2DR_alg} 
	reduces to the vanilla DRS algorithm after a finite 
	number of iterations. By \cite[Theorem 2]{drs_classic}, this means that 
	$\lim_{k\rightarrow\infty} g^k= \lim_{k\rightarrow\infty} v^k-v^{k+1} = \delta v^\star= 0$.
	Thus, we always have  
	$\liminf_{k\rightarrow\infty}\|g^k\|_2=0$, and this fact coupled with 
	Lemma \ref{fp_prob_res} immediately gives us \eqref{liminf}.
	
	Notice that the case of finite $k_i$'s cannot actually happen. 
	Otherwise, since 
	$\lim_{k\rightarrow\infty}\|g^k\|_2=0$ and $n_{AA}$ is upper bounded 
	(because AA candidates are rejected after some point), the check 
	on line 14 of Algorithm \ref{A2DR_alg} must pass eventually. 
	This means that an AA candidate is accepted one more time, 
	which is a contradiction. Hence it must be that AA candidates 
	are adopted an infinite number of times.

	\item \textbf{Case (ii) [Theorem \ref{global_conv_well}, 
		iteration convergence]}
	
	Now suppose that $F_{\rm DRS}$ has a fixed point. As $G_{\rm DRS}$ 
	is non-expansive, if the AA candidate is adopted in iteration $k$,
	\begin{equation*}
	\begin{split}
	\|g^{k+1}\|_2&=\|G_{\rm DRS}(v^{k+1})\|_2\leq 
	\|G_{\rm DRS}(v^{k+1})-G_{\rm DRS}(v^k)\|_2+\|G_{\rm DRS}(v^k)\|_2\\
	&\leq (\|H_k\|_2+1)\|g^k\|_2\leq 2(1+1/\eta)\|g^k\|_2,
	\end{split}
	\end{equation*}
	where we have used Lemma \ref{Hkbound} to bound $\|H_k\|_2$. 
	This immediately implies that for any $0\leq K\leq R-1$, 
	\begin{equation}\label{gk-bound}
	\|g^{k_i+K}\|_2\leq (2+2/\eta)^K\|g^{k_i}\|_2\leq 
	D\|g^0\|_2(2+2/\eta)^K(i+1)^{-(1+\epsilon)},
	\end{equation}
	and so we have $\lim_{i\rightarrow\infty}\|g^{k_i+K}\|_2=0$. 
	
	In addition, since AA candidates are accepted in all iterations 
	$k_i+K$, again by Lemma \ref{Hkbound}, 
	we have that for any $w\in\reals^n$,
	\begin{equation}\label{AA-cand}
	\begin{split}
	\|v^{k_i+K+1}-w\|_2&\leq \|v^{k_i+K}-w\|_2+(1+2/\eta)
	\|g^{k_i+K}\|_2\\
	&\leq \cdots\leq \|v^{k_i}-w\|_2+(1+2/\eta)\sum_{j=0}^K
	\|g^{k_i+j}\|_2\\
	&\leq \|v^{k_i}-w\|_2+(1+2/\eta)\|g^{k_i}\|_2
	\sum_{j=0}^K (2+2/\eta)^j\\
	&\leq \|v^{k_i}-w\|_2+(1+2/\eta)C_RD
	\|g^0\|_2(i+1)^{-(1+\epsilon)},
	\end{split}
	\end{equation}
	where $C_R=\sum_{j=0}^{R-1} (2+2/\eta)^j$ is a constant.

	Now let $v^\star$ be a fixed point of $F_{\rm DRS}$. 
	Since $F_{\rm DRS}$ is $1/2$-averaged, 
	by inequality (5) in \cite{MonoPrimer},
	\begin{equation}\label{DRS-cand2}
	\|v^{l_i+1}-v^\star\|_2^2\leq \|v^{l_i}-v^\star\|_2^2-
	\|g^{l_i}\|_2^2\leq \|v^{l_i}-v^\star\|_2^2
	\end{equation}
	for any $i\geq 0$. Hence for any $k\geq 0$,
	\[
	\|v^k-v^\star\|_2\leq \|v^0-v^\star\|_2 + (1+2/\eta)
	C_RD\|g^0\|_2
	\sum\nolimits_{i=0}^{\infty}(i+1)^{-(1+\epsilon)}=E<\infty,
	\]
	implying that $\|v^k-v^\star\|_2$ is bounded.
	
	As a result, by squaring both sides of \eqref{AA-cand} and 
	combining with \eqref{DRS-cand2}, we get that 
	\[
	\sum_{i=0}^{\infty}\|g^{l_i}\|_2^2\leq \|v^0-v^\star\|_2^2
	+\text{const},
	\]
	where 
	\begin{equation*}
	\begin{split}
	\text{const}=&\left((1+2/\eta)C_RD
	\|g^0\|_2\right)^2\sum_{i=0}^{\infty}(i+1)^{-(2+2\epsilon)}\\
	&+(2+4/\eta)C_RDE\|g^0\|_2\sum_{i=0}^{\infty}
	(i+1)^{-(1+\epsilon)}<\infty.
	\end{split}
	\end{equation*}
	Thus, $\lim_{i\rightarrow\infty}\|g^{l_i}\|_2=0$. Together 
	with the fact that $\lim_{i\rightarrow\infty}\|g^{k_i+K}\|_2=0$ 
	for $0\leq K\leq R-1$, we immediately obtain $\lim_{k\rightarrow\infty}
	\|g^k\|_2=0$, and an application of Lemma \ref{fp_prob_res} yields \eqref{liminf}.
	
	Notice that in our derivation, we implicitly assumed both index sets are infinite. 
	The set of $k_i$ is always infinite by the same logic as in case (i).
	Moreover, if the set of $l_i$ is finite, the arguments above involving $l_i$ can be ignored, 
	as eventually $k = k_i + K$ for all $i$ above some threshold.
	
	
	It still remains to be shown that $v^k$ converges to a fixed-point of $F_{\rm DRS}$. To do this, we first show that $\|v^k - v^\star\|_2$ is quasi-Fej\'erian. Squaring both sides of the first inequality in \eqref{AA-cand}
	and combining it with \eqref{gk-bound} and \eqref{DRS-cand2}, we get that for any $k\geq 0$, 
	\begin{equation}\label{quasi-fejer}
	\|v^{k+1}-v^\star\|_2^2\leq \|v^k-v^\star\|_2^2+\epsilon^k,
	\end{equation}
	where $\epsilon^{l_i}=0$ and 
	\begin{equation*}
	\begin{split}
	\epsilon^{k_i+K}=&2DE\|g^0\|_2(1+2/\eta)(2+2/\eta
	)^K(i+1)^{-(1+\epsilon)} \\
	&+ \left(D\|g^0\|_2(1+2/\eta)\right)^2(2+2/\eta
	)^{2K}(i+1)^{-(2+2\epsilon)}
	\end{split}
	\end{equation*}
	for $0\leq K\leq R-1$. Hence $\epsilon^k\geq 0$ and 
	$\sum_{k=0}^{\infty}\epsilon^k<\infty$. In other words, 
	$\|v^k-v^\star\|_2$ is quasi-Fej\'erian.
	
	Since $\lim_{k\rightarrow\infty} \|g^k\|_2 = 0$ and inequality \eqref{quasi-fejer} holds,
	we can invoke \cite[Theorem 3.8]{Fejer} to conclude that $\lim_{k\rightarrow\infty} \|v^k-v^\star\|_2$ exists
	and $v^k$ converges to some fixed-point of $F_{\rm DRS}$ (not necessarily $v^\star$).
	The convergence of $x^{k+1/2}$ to a solution of \eqref{general} follows directly from 
	the continuity of the proximal operators.
	
	
	\item \textbf{Case (iii) [Theorem \ref{global_conv_path}]} 
	
	Now suppose that problem \eqref{general} is pathological, 
	then $\delta v^\star\neq 0$. Since 
	\[
	\|\delta v^\star\|_2=\inf_{v\in\mathbf{R}^n}~\|v-F_{\rm DRS}(v)\|_2,
	\]
	the safeguard will always be invoked
	for sufficiently large iteration $k$ because $\|g^k\|_2 \geq \|\delta v^\star\|_2>0$. 
	Hence the algorithm reduces to vanilla DRS in the end. 
	We can thus prove the result in case (iii) by appealing to previous work on vanilla DRS
	\cite{drs_classic, range_DRS, DRS_pathology}.
	
	Recall that $\lim_{k\rightarrow\infty} v^k-v^{k+1} = \delta 
	v^\star\neq 0$ \cite[Theorem 2]{drs_classic}. First, we will show that problem 
	\eqref{general} is dual strongly infeasible if and only if 
	\[
	\lim_{k\rightarrow\infty}Ax^{k+1/2}=b.
	\]
	If the problem is dual strongly infeasible, then by \cite[Lemma 1]{DRS_pathology}, 
	it is primal feasible and has an improving direction $d = -\frac{1}{t}\delta v^\star$ 
	\cite[Corollary 3]{DRS_pathology}. Along this direction, 
	both $f$ and $g=\mathcal{I}_{\{x\,:\,Ax=b\}}$ remain 
	feasible, and in particular, $A\delta v^\star=0$. Hence 
	\[
	\lim_{k\rightarrow\infty}Ax^{k+1/2}-Ax^{k+1}=
	\lim_{k\rightarrow\infty}A(v^k-v^{k+1})=A\delta v^\star=0,
	\]
	which implies that $\lim_{k\rightarrow\infty}Ax^{k+1/2}=b$ 
	since $Ax^{k+1}=b$ for all $k\geq 0$.
	
	Conversely, if $\lim_{k\rightarrow\infty}Ax^{k+1/2}=b$, 
	then $\dist(\dom f, \dom g)=0$ because $x^{k+1/2}\in \dom f$. 
	This implies problem \eqref{general} is not 
	primal strongly infeasible, so it must be dual strongly 
	infeasible since we assumed the problem is pathological.
	
	Hence if $\lim_{k\rightarrow\infty}Ax^{k+1/2}=b$, 
	problem \eqref{general} is dual strongly infeasible, and by 
	\cite[Lemma 1 and Corollary 3]{DRS_pathology}, it is unbounded and 
	\[
	\delta v^\star=t\Pi_{\overline{\dom f^* +\dom g^*}}(0),
	\]
	which implies that 
	\[
	\|\delta v^\star\|_2=t\dist(\dom f^*, \range(A^T)).
	\]
	Otherwise, the problem is not dual strongly infeasible
	and thus must be primal strongly infeasible by our assumption of pathology, 
	so from \cite[Corollary 6.5]{range_DRS}, 
	\[
	\|\delta v\|_2\geq \dist(\dom f,\{x\,:\,Ax=b\}).
	\] 
	When the dual problem is feasible, 
	$\delta v^\star=\Pi_{\overline{\dom f-\dom g}}(0)$ \cite[Corollary 5]{DRS_pathology}, 
	which implies that 
	\[
	\|\delta v^\star\|_2=\dist(\dom f, \{x\,:\,Ax=b\}).
	\]
\end{itemize}

\section{Proof of Theorem 
	\ref{global_conv_well_approx}}

The proof resembles that
of Theorem \ref{global_conv_well} (with identical notation), so here we mainly 
highlight the differences caused by the 
computational errors $\eta_1^k,~\eta_2^k$. We begin by bounding the difference between the 
error-corrupted fixed-point mapping, denoted by $\hat{F}_{\rm DRS}$, and the error-free 
mapping $F_{\rm DRS}$. Starting from any $v^k\in\reals^n$, we have by definition
\[
\|\hat{v}^{k+1/2}-v^{k+1/2}\|_2=2\|\hat{x}^{k+1/2}
-x^{k+1/2}\|_2=2\|\eta_1^k\|_2,
\]
\[
\|\hat{x}^{k+1}-x^{k+1}\|_2\leq \|\hat{v}^{k+1/2}
-v^{k+1/2}\|_2+\|\eta_2^k\|_2=2\|\eta_1^k\|_2+
\|\eta_2^k\|_2,
\]
where the inequality comes from the non-expansiveness of $\Pi$. Let $\hat{G}_{\rm DRS}(v) = v - \hat F_{\rm DRS}(v)$. 
Since $\|\eta_1^k\|_2 \leq \epsilon'$ and $\|\eta_2^k\|_2 \leq \epsilon'$,
\begin{equation*}
\begin{split}
\|g^k-G_{\rm DRS}(v^k)\|_2 &= \|\hat{G}_{\rm DRS}(v^k)-G_{\rm DRS}(v^k)\|_2 \\
&= \|\hat{F}_{\rm DRS}(v^k)-F_{\rm DRS}(v^k)\|_2 \\
&\leq \|\hat{x}^{k+1}-x^{k+1}\|_2+\|\hat{x}^{k+1/2}- x^{k+1/2}\|_2 \\
&\leq 3\|\eta_1^k\|_2+\|\eta_2^k\|_2 \leq 4\epsilon'.
\end{split}
\end{equation*}
Thus, by Lemma \ref{fp_prob_res}, it suffices to prove that 
$\liminf_{k\rightarrow\infty} \|G_{\rm DRS}(v^k)\|_2\leq 4\epsilon'
+4\sqrt{L\epsilon'}$.




On the one hand, if the set of $k_i$ (AA candidates) is infinite,
\begin{equation*}
\begin{split}
\liminf_{k\rightarrow\infty}\|G_{\rm DRS}(v^k)\|_2&\leq 
\liminf_{i\rightarrow\infty}\|G_{\rm DRS}(v^{k_i})\|_2\leq 
\liminf_{i\rightarrow\infty}\|g^{k_i}\|_2+4\epsilon'\\
&\leq D\|g^0\|_2\lim_{i\rightarrow\infty}(i+1)^{-(1+\epsilon)}
+4\epsilon'=4\epsilon'.
\end{split}
\end{equation*}
Otherwise, the set of $k_i$ is finite, and the algorithm 
reduces to vanilla DRS after a finite number of 
iterations. Without loss of generality, suppose we start 
running the error-corrupted vanilla DRS algorithm from the 
first iteration. 

Let $v^\star$ be a fixed-point of $F_{\rm DRS}$. 
By inequality (5) in \cite{MonoPrimer},
\begin{equation}\label{DRS-cand3}
\begin{split}
\|v^{k+1}-v^\star\|_2^2&\leq \left(\|\hat{F}_{\rm DRS}(v^k)-
F_{\rm DRS}(v^k)\|_2+\|F_{\rm DRS}(v^k)-v^\star\|_2\right)^2\\
&\leq 16(\epsilon')^2+8\epsilon'\|v^k-v^\star\|_2+\|F_{\rm DRS}(v^k)
-v^\star\|_2^2\\
&\leq 16(\epsilon')^2+16L\epsilon'+ \|v^k-v^\star\|_2^2-
\|G_{\rm DRS}(v^k)\|_2^2
\end{split}
\end{equation}
for all $k \geq 0$, where in the second step, we use the fact that 
$\|\hat{F}_{\rm DRS}(v^k) - F_{\rm DRS}(v^k)\|_2 \leq 4 \epsilon'$ and $F_{\rm DRS}$ is non-expansive, and in the third step, we employ $\|v^k\|_2\leq L$ and $\|v^\star\|_2\leq L$ along with the triangle inequality.
Rearranging terms and telescoping the inequalities,
\[
\dfrac{1}{K}\sum_{k=0}^{K-1}\|G_{\rm DRS}(v^k)\|_2^2\leq 
\dfrac{1}{K}\|v^0-v^\star\|_2^2+16(\epsilon')^2+16 L\epsilon',
\]
which immediately implies that 
\[
\liminf_{k\rightarrow\infty}\|G_{\rm DRS}(v^k)\|_2\leq \sqrt{16(\epsilon')^2+16 L\epsilon'} \leq 4\epsilon' 
+4\sqrt{L\epsilon'}.
\]
Together with Lemma \ref{fp_prob_res}, this completes the proof.

\end{document}